\documentclass{article} 

\usepackage[english]{babel} 

\usepackage[noend]{algpseudocode}
\usepackage[algo2e,ruled,linesnumbered,boxed,boxruled,vlined]{algorithm2e}
\usepackage{arxiv}
\usepackage{amssymb}
\usepackage{amsmath}
\usepackage{amsthm}
\usepackage{mathtools}
\usepackage{bm}
\usepackage[mathscr]{euscript}
\usepackage{url}

\usepackage{abstract}

\newtheorem{theorem}{Theorem}

\newtheorem{definition}[theorem]{Definition}
\newtheorem{lemma}[theorem]{Lemma}
\newtheorem{corollary}[theorem]{Corollary}
\newtheorem{proposition}[theorem]{Proposition}
\newtheorem{remark}{Remark}
\newtheorem{notation}[theorem]{Notation}
\newtheorem{example}[theorem]{Example}
\newtheorem{claim}[theorem]{Claim}
\newtheorem{problem}{Problem}

\DeclarePairedDelimiter{\norm}{\lVert}{\rVert}
\DeclarePairedDelimiter{\abs}{\vert}{\vert}

\usepackage{xcolor}
\usepackage{pgfplots}
\usepackage{bbm}

\usepackage[shortlabels]{enumitem}
\usepackage{graphicx}
\usepackage{todonotes}
\usepackage{hyperref}
\usepackage{listings}
\usepackage{bm}

\newcommand{\Gfab}{\cG_f(\bma,\bmb)}

\newcommand{\cG}{\mathcal{G}}
\newcommand{\RR}{\mathbb{R}}
\newcommand{\CC}{\mathbb{C}}
\newcommand{\ZZ}{\mathbb{Z}}
\newcommand{\NN}{\mathbb{N}}
\newcommand{\QQ}{\mathbb{Q}}
\newcommand{\KK}{\mathbb{K}}

\newcommand{\cZ}{\mathcal{Z}}

\renewcommand{\NN}{\mathbb{N}}
\newcommand{\VV}{\mathbb{V}}
\newcommand{\PP}{\mathbb{P}}
\newcommand{\cI}{\mathcal{I}}
\newcommand{\VVK}{\VV_{\KK}}
\newcommand{\VVKT}{\VV_{\KK^{*}}}

\newcommand{\scX}{\mathscr{X}}

\definecolor{OliveGreen}{rgb}{0,0.6,0}
\newcommand{\image}{\mathtt{Im}}

\newcommand{\NP}{\mathsf{NP}\xspace}

\newcommand{\Kinf}{\ensuremath\mathcal{K}_{\infty}}
\newcommand{\Kall}{\ensuremath\mathcal{K}}

\newcommand{\lcoeff}{\mathtt{lc}\xspace}
\newcommand{\x}{\bm{x}}
\newcommand{\bma}{\bm{a}}
\newcommand{\bmb}{\bm{b}}

\newcommand{\y}{\bm{y}}

\newcommand{\p}{\bm{p}}

\newcommand{\bj}{\bm{j}}

\def\set#1{\{ #1 \}}
\def\Set#1{\left\{ #1 \right\}}
\def\Bigbar#1{\mathrel{\left|\vphantom{#1}\right.}}
\def\Setbar#1#2{\Set{#1 \Bigbar{#1 #2} #2}}

\def\Bigbar#1{\mathrel{\left|\vphantom{#1}\right.}}
\def\Setbar#1#2{\Set{#1 \Bigbar{#1 #2} #2}}

\renewcommand{\emptyset}{\varnothing}

\DeclareMathOperator{\cdim}{cdim}				%
\DeclareMathOperator{\Crit}{Crit}				%
\DeclareMathOperator{\conv}{conv}
\DeclareMathOperator{\supp}{supp}				%

\newcommand{\bdel}{\bm{\Delta}}					%
\newcommand{\bgam}{\bm{\Gamma}}					%
\newcommand{\unze}{\bm{0}}				%
\DeclareMathOperator{\grad}{grad}						%
\newcommand{\cBf}{\mathcal{B}_f}%
\newcommand{\cBfz}{\mathcal{B}^0_f}%
\newcommand{\cBfi}{\mathcal{B}^{\infty}_f}%
\newcommand{\cKf}{\mathcal{K}_f}%
\newcommand{\cK}{\mathcal{K}}%
\newcommand{\cKfz}{\mathcal{K}^0_f}%
\newcommand{\cKfi}{\mathcal{K}^{\infty}_f}%
\newcommand{\cKFi}{\mathcal{K}^{\infty}_F}%
\newcommand{\bmf}{\bm{f}}%
\newcommand{\fstar}{\left. f\right|_{*}}%
\newcommand{\Rgf}{\mathcal{D}_\Gamma(f)}
\newcommand{\Dbg}{D_{\bgam}}%
\newcommand{\Ibdel}{\mathcal{I}(\bdel)}
\newcommand{\Obdel}{\mathcal{O}(\bdel)}
\newcommand{\IObdel}{\mathcal{I}\!\mathcal{O}(\bdel)}
\newcommand{\bmg}{\bm{g}}
\newcommand{\bmgI}{\bmg_I}

\DeclareMathOperator{\Jac}{Jac}

\DeclareMathOperator{\Proj}{Proj}
\newcommand{\f}{\bm{f}\xspace}							%
\newcommand{\g}{\bm{g}\xspace}
\newcommand{\Rnn}{\R^n_{\geq 0}}
\newcommand{\R}{\mathbb{R}}
\newcommand{\C}{\mathbb{C}}
\newcommand{\Z}{\mathbb{Z}}
\newcommand{\N}{\mathbb{N}}
\newcommand{\Q}{\mathbb{Q}}
\newcommand{\K}{\mathbb{K}}
\newcommand{\cR}{\mathcal{R}}
\newcommand{\V}{\mathbb{V}}

\newcommand{\TTnK}{(\K^*)^n}

\newcommand{\cB}{\mathcal{B}}

\newcommand{\VsK}{\mathbb{V}^{*}_{\!\K}}

\newcommand{\FrS}{\left.F\right|_{S}}
\newcommand{\Fr}{\left.F\right|}
\newcommand{\fr}{\left.f\right|}

\newcommand{\VK}{\mathbb{V}_{\!\K}}
\newcommand{\VC}{\mathbb{V}_{\!\C}}
\newcommand{\VR}{\mathbb{V}_{\!\R}}

\newcommand{\cS}{\mathcal{S}}
\newcommand{\cC}{\mathcal{C}}

\def\norm#1{\| #1 \|}
\def\normi#1{\norm{#1}_{\infty}}

\newcommand{\OB}{\mathcal{O}_B}
\newcommand{\OO}{\mathcal{O}}
\newcommand{\sOB}{\widetilde{\mathcal{O}}_B}
\newcommand{\sOO}{\widetilde{\mathcal{O}}}

\newcommand{\balpha}{\bm{\alpha}}

\title{Bounds on the infimum of polynomials over a generic semi-algebraic set using asymptotic critical values}

\author{
	Boulos El Hilany  
	\thanks{Institut f\"ur Analysis und Algebra, TU Braunschweig, Germany,
	 {boulos.hilani@gmail.com, b.el-hilany@tu-braunschweig.de}}
	\and
	\textbf{Elias Tsigaridas}
	\thanks{Inria Paris and Institut de 
    Math{\'e}matiques de Jussieu--Paris Rive Gauche, Sorbonne
    Universit\'e and Paris Universit\'e, {elias.tsigaridas@inrial.fr}}
}

\begin{document}

\maketitle
	
\begin{abstract}
We present precise bit and degree estimates
for the optimal value of the polynomial optimization problem 
$f^*:=\inf_{x\in \scX}~f(x)$, 
where $\scX$ is a semi-algebraic set satisfying some non-degeneracy conditions.
Our bounds depend on the degree, the bitsize of $f$, and the polynomials defining $\scX$, 
and are single exponential with respect to the number of variables.
They generalize the single exponential bounds from Jeronimo, Perrucci, and Tsigaridas 
(SIAM Journal on Optimization, 23(1):241–255, 2013)
for the minimum of a polynomial function on a compact connected component of a basic closed semi-algebraic set.

The tools that we use allow us to obtain specialized bounds and dedicated algorithms for two
large families of polynomial optimization problems
in which the optimum value might not be attained. The first family forms a dense set of real polynomial functions with a fixed collection of Newton polytopes; we provide the best approximation yet for the bifurcation set, which contains the optimal value, and we deduce an effective method for computations. As for the second family, we consider any unconstrained polynomial optimization problem; we present more precise bounds, together with a better bit complexity estimate of an algorithm to compute the optimal value. 
	

\end{abstract}

\maketitle
\newpage
\tableofcontents
\newpage

\section{Introduction}

Let $f:\scX\longrightarrow \R$ be a semi-algebraic function over a basic semi-algebraic set $\scX\subset\R^n$.
We consider the problem of computing effective bounds on the infimum $f^*:=\inf_{x\in \scX}f(x)$, assuming it lies in $\R$, as a function of the number of variables, the number of polynomials defining the semialgebraic set, their degrees and bitsize. 
\begin{problem}\label{prb:main}
The optimal value of the polynomial optimization problem 
$f^* = \inf_{x\in \scX}f(x)$ is a real algebraic number, 
when the defining polynomials have rational coefficients. 
Provide an 
effective lower bound on its value
and an upper bound on degree of the defining polynomial
that depend on the number of variables $n$, the degrees, and the coefficient (bit)size of $f$ and the polynomials defining $\scX$.
\end{problem}

The computation of effective, explicit, and if possible (asymptotically) optimal bounds on the size and degree of the optimal value is an important problem on its own,
as these are intrinsic quantities of the polynomial optimization problem that characterize and measure its difficulty. 
Furthermore, they are fundamental quantities for the analysis of various symbolic and numerical algorithms,
with a wide range of applications, e.g.,~\cite{song2017low,schaefer2017fixed,burr2020complexity,la2024some}. As the optimal value is an algebraic number, the degree of the defining polynomial is essential in complexity statements, as it measures the algebraic complexity of the problem, e.g., \cite{draisma2016euclidean,bajaj1988algebraic}.


If the semi-algebraic subset $\scX$ is compact, then the infimum of $f$ is attained at $\scX$.
In this case, explicit bounds for Problem~\ref{prb:main} 
and an effective method for computing the infimum already exist. 
 Namely, if $\scX$ is defined by $r$ equalities, $s$ inequalities with polynomials in $\Z[x_1,\ldots,x_n]$, having degrees at most $d$, and coefficients of absolute values at most $H$, then either $f^*$ is not reached over $\scX$, or is an algebraic number of degree at most 
 \begin{equation}
 \label{eq:alg_deg_attained}
 	\max_{0 \leq i \leq \min\{r+s, n\}} {n \choose i } d^i (d-1)^{n-i} 
 \leq 2^{n-1} d^n,
 \end{equation}
 and if it is not zero, then 
\begin{equation}
	\label{eq:bounds_attained}
	|f^*|\geq  (2^{5 -\frac{n}{2}} \tilde{H}d^n)^{- n2^nd^n},
\end{equation} 
where $\tilde{H}:=\max(H,~n+r+s)$~\cite[Theorem 1.1]{jpt-min-siopt-2013}.
If we let $H \leq 2^{\tau}$, then 
$\abs{f^{*}} \geq 2^{-\sOO(n d^n \tau)}$.
We notice that the bound is single exponential with respect to the number of variables; hence polynomial when the number of variables is fixed.
Even more, it is (asymptotically) optimal.
Jeronimo et al \cite{jpt-min-siopt-2013} proved the inequality in \eqref{eq:bounds_attained}
by bounding the isolated roots of carefully deformed polynomial systems; this was the first precise bound for the optimum value of the polynomial optimization problem.
While the initial motivation in \cite{jpt-min-siopt-2013} was to approximate the minimum distance between two semi-algebraic sets, 
the bounds has already several other applications, e.g., \cite{song2017low,schaefer2017fixed}.
Let us also mention that there is also an efficient (symbolic) algorithm 
to compute it due to Jeronimo and Perrucci \cite{jeronimo2014probabilistic}. 

While the bound of \eqref{eq:bounds_attained} is quite important, 
the assumption that the optimum value is attained is rather restrictive in several problems and applications.
Not all semi-algebraic functions attain their infima; as an example, consider the 
infimum of the function $x\mapsto 1/x$ over $\{x>0\}$ is $0$.
We make a step towards closing this gap in the literature
by considering the Problem~\ref{prb:main} where $f$ admits a non-trivial infimum $f^*\in\R$, but it is not attained on $\scX$;
under some assumptions on $\scX$ that we detail in the sequel.
The computation of the optimum value, both in the \emph{"attained"} and in the \emph{"unattained"} case, 
	appears frequently in science and engineering. 
	Namely, systems and problems in applications, e.g., computer aided design~\cite{farin2002curves}, robotics~\cite{siciliano2009robotics}, control systems~\cite{khalil2002nonlinear}, to mention a few, are modeled as semi-algebraic subsets,
	and so a recurring task is to decide whether two semi-algebraic subsets are disjoint, e.g.,~\cite{schwartz1983piano}. 
One approach to solve this problem is to figure out whether 
the images (under the same function) of the two semi-algebraic sets are disjoint or not. Accordingly, if the infimum of one function is larger than the maximum of the other, the two semi-algebraic sets are disjoint.

A 
straightforward algorithm to compute 
the optimum value of Problem~\ref{prb:main} is 
to use (the general purpose approach of) cylindrical algebraic decomposition (CAD)
 for stratifying semi-algebraic projections, e.g.~\cite{caviness2012quantifier}. 
However, this approach requires us to compute many algebraic objects that are possibly unnecessary, as it relies on repeated projections. Even more, it has a worst-case complexity that is double exponential in the number of variables~\cite{bpr-06,bd-cad-bound}. It will also result double exponential bounds
for the optimal value(s), $f^*$. 
Another approach consists in describing the problem as 
a general decision problem for the theory of reals
and then perform quantifier elimination \cite{BPR03}.
However, it is not clear, at least to us, what is the optimal formulation in this setting and how to obtain precise bit and degree estimates.
 We opt for a different approach.

We provide effective bounds on the infimum $f^*$ of a semi-algebraic function $\scX\longrightarrow\R$, when the polynomials involved in $\scX$ satisfy a non-degeneracy condition. 
The bound is single exponential with respect to the number of variables (Theorem~\ref{thm:main}). Up to the non-degeneracy assumptions, our contribution fills the gap in the literature for Problem~\ref{prb:main}
and to the best of our knowledge is the first bound for when the optimal value is not attained.

Our approach builds on the works of Jelonek and Kurdyka \cite{jelonek2005quantitative,JK03}
and Jelonek and Tib{\u{a}}r \cite{JelTib17} for approximating the \emph{bifurcation set at infinity}, $\cBfi$ of a polynomial function $f$~\cite{Tho69,Tib07,vui2008critical}. The set $\cBfi$ is the smallest set of values, outside of which $f$ is a locally trivial $\cC^{\infty}$-fibration ``at infinity'' (i.e., outside a large ball), and has been a subject of intensive study in the last fifty years, e.g., \cite{Tib07}. Clearly, the set $\cBfi$ contains the infimum $f^*$ of a real polynomial function. Consequently, and for a plethora of other reasons, it remains an important open problem to derive an efficient algorithm to compute $\cBfi$ for any polynomial function $f$~\cite{Tib07,vui2008critical}. The results of Jelonek, Kurdyka, and Tib{\u{a}}r
	lead to an efficient computation of the set of points, $\cKfi$, called \emph{Rabier set} of $f$, or the \emph{asymptotic critical values} of $f$, as
\begin{align}\label{eq:bifur-Rabier}
\cKfi:=& ~\big\{z\in\R~|~\exists \{\x_\ell\}_{\ell\in\N}\subset X,~\Vert \x_\ell\Vert\underset{\ell\rightarrow\infty}{\longrightarrow}\infty,~\Vert \x_\ell\Vert\cdot \Vert\grad f(\x_\ell)\Vert\longrightarrow 0,~f(\x_\ell)\longrightarrow z \big\}.
\end{align} 
This is a finite set that contains $\cBfi$,
that is $\cBfi \subseteq \cKFi$~\cite{rabier1997ehresmann}.
Therefore, these important results provide new computational tools for $f^*$. 
There are several related methods dedicated in computing $\cBfi$ for large families of polynomial functions (e.g.,~\cite{Nem86,Zah96,pham2016genericity}).
We build on methods that require the fewer possible assumptions on the input (polynomial) functions. 

One of our contributions is generalize the previous techniques from polynomial functions to semi-algebraic functions.
	In this way, we derive the best known estimates for bounds on $f^*$, and develop algorithms to compute $f^*$ in two important special cases of Problem~\ref{prb:main}.

The first (special) case 
considers the semi-algebraic set $\scX$ of Problem~\ref{prb:main} to be 
a real affine variety that satisfies a mild non-degeneracy condition. 
	This condition (Theorem~\ref{thm:face-generic_tuple}) applies for a dense family of functions sharing the same collection of Newton polytopes and gives rise to functions having a non-trivial bifurcation set. 
We present single exponential bit and degree bounds for $f^*$
and we introduce an algorithm to compute it with precise (single exponential) bit complexity estimates.
This algorithm computes directly the bifurcation set and exploits the collection of Newton polytopes of the input polynomials.
	In this way, we avoid to compute the larger set of critical values at infinity
	and we also exploit the sparsity of the input polynomials. To the best of our knowledge, prior to our work, there was no dedicated algorithm to compute the infimum of real polynomial functions above.	

The second (special) case 
is the unconstrained polynomial optimization problem, that is when $\scX = \R^n$. An effective algorithm for approximating $\cKfi$ (together with its arithmetic complexity analysis, and a method for computing $f^*$) was developed by Safey El Din~\cite{Safey-opt-R-2008}. As was pointed out by Jelonek and  Tib{\u{a}}r~\cite{JelTib17}, this would only provide a subset of $\cKfi$. We present single exponential (bit and degree) bounds for $f^*$ and a probabilistic algorithm to compute it, that is based on \cite{JK03} and makes no assumptions on the input (Theorem~\ref{thm:unconstrained_detailed}) along with precise (single exponential) bit complexity estimates. 
Even though a method for computing $f^*$ exists~\cite{Safey-opt-R-2008}, to the best of our knowledge, neither bit/degree bounds were known before, nor precise bit complexity estimates for an algorithm to compute it;
without any assumptions on the input.

In the rest of this section we present in detail the theorems that support our bounds for $f^*$ and the corresponding algorithms, along with a bird's eye view of the proof techniques that we employ.


\subsection{Presentation of the main results}

We denote by $\OO$, resp. $\OB$, the arithmetic, resp.  bit, complexity and we use the soft-O notations,  $\sOO$, respectively $\sOB$, to ignore (poly-)logarithmic factors.
We refer to~\S~\ref{sec:notation} for further details on the notation that we use.

Let $\scX\subset\R^n$  be a semi-algebraic 
defined by a set of polynomials $\bm{f} \subset\R[x_1,\ldots,x_n]$. 
We call $\scX$ \emph{complete} 
if it is closed, connected, and for each subset $\bm{g} \subseteq \bm{f}$,
the variety $\VV(\bm{g}) = \Setbar{\x \in \CC^n}{g(\x) = 0, \text{ for all } g \in \bm{g}} \subset \CC^n$ is a complete 
intersection and smooth;
the latter conditions means that the Jacobian of $\bm{g}$ evaluated at the points of $\VV(\bm{g})$ has full rank.

The following theorem provides bit and degree bounds for the infimum of
Problem~\ref{prb:main} when $\scX$ is complete. 


\begin{theorem}\label{thm:main}
Let $\scX$ be a complete semi-algebraic set given by polynomial (in)equalities 
	\[
\scX: = \set{g_1=\cdots=g_r=0,~g_{r+1},\ldots,g_s\geq 0,~g_1,\ldots g_s\in\R[x_1,\ldots,x_n]} ,
\]
 let $F:\R^n\longrightarrow\R$ be a polynomial function of degree $d_1$, and assume that the infimum $f^*$ of the function $f:=\Fr_{\scX}:\scX\longrightarrow\R$ is not attained. 
 	Then $f^*$ is an algebraic number of degree $\OO((n r d_1)^{n^2})$, such that 
	$2^{-\eta} \leq \abs{f^*} \leq 2^{\eta}$, where $d=\max\{\deg g_1,\ldots,\deg g_s\}$ and 
	\begin{align*} 
	\eta:= &~\sOO(r(d + n + \tau) + d_1)(n r d_1)^{n^2} ).
\end{align*}	
\end{theorem}

If we compare the bounds of the minimum 
induced by Eq.~\eqref{eq:alg_deg_attained} and \eqref{eq:bounds_attained} 
with the bounds on the infimum of the previous theorem,
then we notice that both bounds are single exponential with respect to the number of variables. 
However, the bound of Theorem~\ref{thm:main} admits a worse exponent, $n^2$ instead of $n$. 
It is worth noting that the dependence, in both cases, in the number of polynomials and the degrees is polynomial.




To give an overview of the approach that we follow to prove Theorem~\ref{thm:main}, let $X\subset \R^n$ be a real affine variety.
Then, for every $z\in\R$ close enough to $f^*$, the preimage $f^{-1}(z)$ is empty if $z< f^*$. Hence, the infimum $f^*$ is a point in $\cBfi$. 
Since we have the inclusion
\begin{equation}
	\label{eq:Bf-in-Kf}
	\cBfi\subseteq\cKfi ,
\end{equation} we can compute $\cKfi$ instead.
Notice that the previous inclusion can be sharp for some functions~\cite[Example 3.8]{vui2008critical}.

A straightforward analogue of the bifurcation set at infinity does not exist for arbitrary semi-algebraic sets $\scX$. 
 Because of this, we decompose $\scX$ from Theorem~\ref{thm:main} into finitely-many semi-algebraic subsets and consider their real Zariski closures, which we assume to be 
  smooth and irreducible. For each such algebraic $X\subset \R^n$, we show that the infimum $f^*$ belongs to the bifurcation set of the restricted polynomial function $\Fr_X$, where $F$ is the canonical extension of $f$ to the space $\R^n$ (Theorem~\ref{thm:bifur-semi}). Then, we can approximate the bifurcation set at infinity by the Rabier set of $\Fr_X$. This way, we can then use effective methods developed in~\cite{jelonek2005quantitative}, to compute a superset of the Rabier set.


Using the above description, one can extract lower bounds on the infimum; roughly speaking, following \cite{jelonek2005quantitative}, 
	we express the Rabier set as the intersection of (the closure of) the image of a polynomial map with a linear subspace. This results in a univariate polynomial whose real roots include the infimum. Since the number of polynomials in the process is much larger than the number of variables, we should avoid using Gr\"obner basis computations, 
	because it is difficult to bound their (bit) complexity
	and, even worse, they might induce double exponential bounds on the degree and bitsize of the resulting polynomial. Instead, we use resultant systems, e.g.~\cite{vdWaerden-alebra-37,Yap-algebra-book}
	which can express the image of the polynomial map using minors of a 
	Macaulay-type matrix whose size exploits the single exponential bounds of the effective Nullstellensatz.
	In this way, we bound the degree and bitsize of the univariate polynomial that has the infimum among its real roots.
A straightforward algorithm based on this approach still 
	has double exponential complexity. However, our goal is not to compute the infimum, but to bound it.


The algorithms  tailored for $\cKfi$ \cite{jelonek2005quantitative} have prohibitive complexity, 
mainly because the require the computation of exponentially many minors of a Jacobian matrix (see~\S\ref{sec:constrained-opt} for details).
This will result in double exponential complexity bound for an algorithm to compute $f^*$.
	However, if we impose some conditions on the semi-algebraic function, then might be able to overcome this computational obstacle. 
Indeed we present two such important special cases where this is possible
and where not only we present improved (compared to Theorem~\ref{thm:main}) lower bounds,
but also efficient algorithms for $f^*$, supported by precise bit complexity estimates.	
	
\subsubsection*{Unconstrained polynomial optimization}
Consider $\scX=\R^n$.  Then, we can approximate the Rabier set by considering the \emph{complexification} $\C f:\C^n\longrightarrow\C$ of the polynomial function $f:\R^n\longrightarrow\R$, which simply extends the domain of $f$. Jelonek and Kurdyka in~\cite{JK03} presented one of the first methods for computing $\cKfi$, while there are also other similar approaches and implementations~\cite{jelonek2005quantitative,safey2007testing,JelTib17}. 
Based on \cite{JK03}, we present an algorithm (Alg.~\ref{alg:ACV}) in~\S\ref{sec:uncon-opt},
where we exploit resultants instead of Gr\"obner basis to perform algebraic elimination.
 Note that, once we compute $\cKfi$, then we can identify the infimum by testing the non-emptiness of fibers in generic points of the intervals $\R\setminus\cKfi$~\cite{Safey-opt-R-2008}. 
 A precise analysis of the degenerate conditions
 and bound on the bitsize of the algebraic objects involved 
 in the computation of Alg.~\ref{alg:ACV} lead to the following theorem for bounding and computing $f^*\in \cKfi$. 

\begin{theorem}[Theorem~\ref{thm:unconstrained_detailed}]\label{thm:main_unconstrained}
	Let $f \in \ZZ[x]$ be of degree $d$ and bitsize $\tau$. Then, 
	$f^{*} = \min_{\x \in \RR^n} f(\x)$ is an algebraic number of 
	degree $\OO(d^{n-1})$, such that
\begin{equation*}
2^{-\eta} \leq \abs{f^*} \leq 2^{\eta}, \quad \text{ where } 
 \eta := ~\sOO(n d^{n-1}\tau + n^{2})).
\end{equation*} 
There is an randomized algorithm that approximate $f^*$ 
in $\sOB(d^{3n}\lg(\tfrac{1}{\epsilon})(n d^{2n-2}\tau + \lg(\tfrac{1}{\epsilon})))$ bit operations, with probability of success $1 - \tfrac{1}{\epsilon}$.
\end{theorem}

Notice the lower bound on $f^*$ in the unconstrained case, has an exponent $n$ instead of $n^2$. 
This is similar to the bounds in 
Eqs.~\eqref{eq:alg_deg_attained} and \eqref{eq:bounds_attained}
where the infimum is attained, and hence it is asymptotically optimal. 
We emphasize that Theorem~\ref{thm:main_unconstrained} 
holds without any assumptions on the input.
If the infimum is attained, then it is among the critical values of $f$,
which we can compute by solving the system of the partial derivatives of $f$. The complexity of this step is dominated by the complexity of the computation of the asymptotic critical value; hence we do not detail on this further.


\subsubsection*{Newton non-degenerate polynomial functions} 

Whenever $\scX=:X$ is a (real) algebraic set, we present an algorithm for computing
the optimal value $f^*$, under some genericity conditions on the input polynomials with respect to their Newton polytopes. 
We refer the reader to~\S\ref{sub:prelim}
for detailed presentation of polynomial sparsity and Newton polytopes.

A polynomial $f\in\K[x_1,\ldots,x_n]$ is a linear combination 
$\sum c_{\bma}~\x^{\bma}$ of monomials $\x^{\bma}:=x_1^{a_1}\cdots x_n^{a_n}$, where the exponent vectors $\bma$ run over a finite subset $A$ of $\N^n$, and $c_{\bma}\in \K^*$. Then $A$ is the \emph{support} of $f$ and the \emph{Newton polytope}, $\NP(f)$, of $f$ is the convex hull of $A$ in $\R^n$. Given a collection $\bdel:=(\Delta_0,\Delta_1,\ldots,\Delta_k)$ of integer polytopes in $\R^n_{\geq 0}$, we use $\R^{\bdel}$ to denote the space of all tuples of polynomials whose \emph{Newton tuple} is $\bdel$:
\begin{align*}
\R^{\bdel}:= & \left\lbrace (f_0,f_1,\ldots,f_k)~\mid~f_i\in\K[z_1,\ldots,z_n],~\NP(f_i) = \Delta_i,~i=0,1,\ldots,k \right\rbrace.
\end{align*} 

We can identify this space with the space of polynomial functions $X\longrightarrow \R$, with $\x\longmapsto f_0(\x)$, where $X $ is the real algebraic set $\VR(f_1,\ldots,f_k)\subset \R^n$, and $\NP(f_i)=\Delta_i$, $i=0,1,\ldots,k$. 

The following theorem states that, for polynomials belonging to a Zariski open set, we can compute their bifurcation set at infinity $\cBfi$ by exploiting their Newton polytopes.

\begin{theorem}[Theorem~\ref{thm:face-generic_tuple}]\label{thm:bifurcation_Newton}
For every collection $\bdel$ of lattice polytopes in $(\R_{\geq 0})^n$, there is a Zariski open subset $\Omega\subset\K^{\bdel}$, such that for every $\bmf\in\Omega$, corresponding to a real polynomial function $f:X\longrightarrow \R$, the set $\cBfi$ can be computed effectively. Furthermore, the values in $\cBfi$ depend only on the coefficients of the polynomials whose exponent vectors appear at the faces of polytopes in $\bdel$.
\end{theorem}

The polynomial functions belonging to the open set $\Omega$ 
of the previous theorem are called \emph{Newton non-degenerate}. 
Theorem~\ref{thm:face-generic_tuple} supports an algorithm for computing $\cBfi$ that has several advantages compared with other "classical" methods. First, we compute a smaller superset of the bifurcation set at infinity. Secondly, the resulting algorithm applies for  a dense family of polynomial functions sharing a collection of Newton polytopes. Finally, since the theorem and the supporting algorithms depends on the Newton polytopes of the input polynomials, we exploit the sparsity of the input and on our way to estimate the bifurcation set we compute with smaller polynomial systems. Ultimately, we obtain the following result.




\begin{theorem}[\S\ref{sub:bounding_infimum}]\label{thm:main_Newton_non-degenerate}
Let $\bdel$ be a tuple of integer polytopes in $(\R_{\geq})^n$, and let $\bmf\in\R^{\bdel}$ be a Newton non-degenerate element corresponding to a real polynomial function $f:X\longrightarrow\R$. Assume furthermore that all polynomials involved in $X$ have degree at most $d$. Then, 
\[
2^{-\eta} \leq \abs{f^{*}} \leq 2^{\eta},
\quad \text{ where } 
\eta := ~\OO(n^2 d^{n-1}(n + \tau)).
\]  
\end{theorem}

Similarly to Theorem~\ref{thm:main_unconstrained}, the bound on the infimum for Newton non-degenerate functions is single exponential in $n$ and matches the bounds in Eqs.~\eqref{eq:alg_deg_attained} and \eqref{eq:bounds_attained}; hence it is asymptotically optimal.

\subsection{Organization of the paper and proof strategies}


The rest of the paper is structured as follows. The next section presents (some of) the notation that we use in throughout and some necessary preliminaries. 

In~\S\ref{sec:real-complex}, we will describe a decomposition $\cS(\scX)$ of a complete semi-algebraic set, which we call \emph{algebraic stratification}. We then show Theorem~\ref{thm:bifur-semi}: For any semi-algebraic function $f:\scX\longrightarrow\R$, one can construct a finite set of complex polynomial functions 
$\{f_s:~V_s\longrightarrow\C\}_{s\in\cS(\scX)}$ such that for each $s\in\cS(\scX)$, we have $s\subset V_s$, and $f^*\in\cK_{f_{t}}$ for some stratum $t$ above. Now, since the subsets $V_s$ are algebraic, the values in $
\cK_{f_{s}}$ can be computed effectively thanks to results of Jelonek and Kurdyka in~\cite{jelonek2005quantitative}. These effective expressions will become useful in~\S\ref{sec:constrained-opt} for showing Theorem~\ref{thm:main}.

In~\S\ref{sec:constrained-opt} we consider the constrained optimization problem, Problem~\ref{prb:main}, under the assumption that the feasible region is a \emph{complete} semi-algebraic set. We employ various tools to obtains precise bitsize and degree estimates for the infimum. 

~\S\ref{sec:Newton-non} is devoted to proving Theorem~\ref{thm:main_Newton_non-degenerate}; we first describe the types of coherent faces in the tuple $\bdel$ that are key for computing the bifurcation set. The proof relies on classical results on $A$-discriminants and face-resultants of polynomial tuples over the real and complex fields~\cite{GKZ94,Est10,Est13}. 
Once the expressions of the infimum are well-established in terms of an elimination ideal, we can then prove upper bounds on its absolute value. 
	
	In~\S\ref{sec:uncon-opt} we consider the unconstrained polynomial optimization problem. We present bounds on the infimum, an algorithm to compute it, and precise bit complexity estimates;
 under no assumptions on the input.
 
	We mention auxiliary results that we need for the proofs of various results in the Appendix.

\subsection{Notation}
\label{sec:notation}

We denote by $\OO$, resp. $\OB$, the arithmetic, resp.  bit, complexity and we use the soft-O notations,  $\sOO$, 
respectively $\sOB$, to ignore (poly-)logarithmic factors.
We denote by $\Rnn$ the positive orthant
and we use the abbreviation $[n]$ for $\{1, 2, \dots, n \}$.
We use bold letters to denote vectors.

A Monte Carlo algorithm is a randomized algorithm that its output might not be correct with a certain probability. A Las Vegas randomized algorithm always outputs the correct result, but its runtime is not always the same, even for the same input. 

\paragraph*{Algebraic varieties and smooth maps}\label{sss:maps_fibrations}

Consider $\x = (x_1, \dots, x_n)$ and let $\x_{-i}$ denote all the variables except the variable $x_i$. 
For a polynomial $f \in \KK[x_1,\ldots, x_n] = \KK[\x]$, where $\KK \in \{\RR, \CC\}$, we
denote its zero set by $\VVK(f) \subset \KK^n$ and by $\VVKT(f) \subset \TTnK$ its zero set over the corresponding torus. We use the same notation after replacing $f$ by its bold form $\bmf$ if it is a tuple of polynomials, that is $\bmf = (f_i)_{i\in I}$ for some finite subset $I\subset \N$.

Let $f:X\longrightarrow Y$ be any smooth, or analytic, map between two manifolds. Let $\cG$ be the graph in $X\times Y$ of $f$ and consider the map
\[
\begin{array}{cccc}
	P:= \pi_\cG: &~ \cG & \longrightarrow & Y \\
	& (\x,~\y)& \longmapsto & \y  \enspace ,
\end{array} 
\] 
where $\pi_\cG$ is the restriction of  $\pi:X\times Y\longrightarrow Y$, $(\x,~\y)\longmapsto \y$ on $\cG$. 
We say that $f$ is a \emph{$\cC^{\infty}$-fibration}, if for every $\y\in \pi(\cG)$, the set $Y\times f^{-1}(\y)$ is diffeomorfic to  $P^{-1}(\pi(\cG))$. The map $f$ is said to be a \emph{locally trivial fibration} if, for each $\x\in X$, there exists a neighborhood $U$ of $x$ for which the restricted map $\fr_{U}:U\longrightarrow Y$ is a $\cC^{\infty}$-fibration.

A point $\x\in X$ is said to be a \emph{critical point of $f$} if the \emph{co-rank} $\text{corank} (df)_{\x} := \min (\dim X,~\dim Y) - \text{rank} (df)_{\x}$ is positive. We denote by $\Crit(f)$ the set of critical points of $f$ and by $\cKfz(f)$ the corresponding critical values.

If, instead, $X$ is a smooth variety defined by polynomials $g_1,\ldots,g_p\in\K[x_1,\ldots,x_n]$, then the critical locus of $f$ consists of all points $\x\in X$ for which the rank of the following \emph{Jacobian matrix}
\[
\Jac \f := \begin{bmatrix}{\partial f_1}/{\partial x_1} & \cdots &{\partial f_k}/{\partial x_1} &
{\partial g_1}/{\partial x_1} &\cdots& {\partial g_p}/{\partial x_1}\\
\vdots &  & \vdots &\vdots &  & \vdots  \\
{\partial f_1}/{\partial x_n} & \cdots &{\partial f_k}/{\partial x_n} &
{\partial g_1}/{\partial x_n} &\cdots& {\partial g_p}/{\partial x_n}
	\end{bmatrix}
\]
 is lower than the codimension of $X$.

\paragraph*{About polynomials and their roots}
For a polynomial $f \in \ZZ[\x] = \ZZ[x_1, \dots, x_n]$ its infinity norm
$\normi{f}$ equals the maximum of absolute values of its
coefficients.  The bitsize of a polynomial is the logarithm of its infinity
norm.  We also call the latter the bitsize of the polynomial,
that is a shortcut for  the maximum bitsize of all its coefficients.
A univariate (multivariate) polynomial is of size $(d,\tau)$ when its (total) degree is at
most $d$ and has bitsize $\tau$.
We represent a real algebraic number $\alpha \in \RR$ using the
\textit{isolating interval representation}; it includes a square-free
polynomial, $A$, which vanishes at $\alpha$ and an interval with rational
endpoints that contains $\alpha$ and no other root of $A$.
If $\alpha \in \CC$, then instead of an interval we use a rectangle in $\RR^2$
where the coordinates of its vertices are rational numbers.


\section{Critical values at infinity and bounds on the infimum}\label{sec:real-complex} 

We present a decomposition of a semi-algebraic set, under some 
transversality conditions. Then, we relate the infimum of a polynomial function, say $f$, restricted on this semi-algebraic set,
with the bifurcation set(s) of $f$, restricted to the strata of the decomposition.

\subsection{The bifurcation and Rabier sets} 

Let $X\subset \K^n$ be a smooth affine variety, where $\K\in\{\R,\C\}$, and let $f:X\longrightarrow\K$ be a polynomial function on $X$. That is, if $X=\VV_{\KK}(\cI)$ for some ideal $\cI$, then $f\in \K[x_1,\ldots,x_n]/\cI$.
	A generalization~\cite{wallace1971linear,varchenko1972theorems,
verdier1976stratifications,Tib99} of Thom's result~\cite{Tho69}
indicates that outside a finite set $S\subset\K$, the following restricted function is a $\cC^{\infty}$-fibration:
\begin{align}
	\label{eq:restricted_map}
\left.f\right|_{f^{-1}(\K\setminus S)}: & f^{-1}(\K\setminus S)\longrightarrow\K\setminus S.
\end{align} 
The smallest subset $S\subset\K$ for which the function in~\eqref{eq:restricted_map} is a $\cC^{\infty}$-fibration is the \emph{bifurcation set} of $f$ and we denote it by $\cBf$.
The \emph{bifurcation values} $z\in\cBf$ are of two types:
(i) $z=f(\x)$ for some $\x\in\Crit(f)$;
we denote this set by $\cBfz$ (also $\cKfz$), or 
(ii) $z$ is such that for an arbitrarily large compact subset $K\subset \C^n$ 
and any small disc $D$ containing $z$, the restricted map 
\begin{align*}
\left.f\right|_{{f^{-1}(D)\setminus K}}:f^{-1}(D)\setminus K\longrightarrow D
\end{align*} is not a $\cC^{\infty}$-fibration (see e.g.,~\cite[Definition 2.1]{Tib99}). The set of all such values is denoted by $\cBfi$, and we call it the \emph{bifurcation set at infinity}. 
Then, we get
$\cBfz = \cBfi \cup  \cBfz$.
We notice that the inclusion $\cBfz\subset\cBf$ can be strict. 
For example, the function $(x,~y)\longmapsto x+x^2y$ has no critical points~\cite{Dur98}, whereas $0$ is a bifurcation value as $f^{-1}(0)$ is the only fiber with more than one connected component in $\K^2$. 

Recall the Rabier set defined at the beginning in~\eqref{eq:bifur-Rabier}. We define $\cKf := \cKfi \cup \cKfz$ to be the set of \emph{generalized critical values} of $f$. 
 Consequently, $f$ is locally trivial fibration over the $\K \setminus \cKf$.

In~\S\ref{sec:constrained-opt} and~\S\ref{sec:uncon-opt} 
we compute $\cKf$ using algebraic elimination.

\subsection{A decomposition of a (complete) semi-algebraic set}\label{sub:decomposing}
Consider a basic closed semi-algebraic set
\begin{equation}
	\label{eq:sa-set}
\scX: = \set{g_1=0, \dots, g_r=0,~g_{r+1} \geq 0, \ldots, g_s\geq 0,~g_1,\ldots g_s\in\R[x_1,\ldots,x_n]},
\end{equation} 
and let $f:\scX\longrightarrow\R$ be a semi-algebraic function given as the restriction  of a polynomial function $F:\R^n\longrightarrow\R$ on $\scX$, that is 
\begin{align*}
	f:= ~\Fr_{\scX}:&~\scX\longrightarrow\R.
\end{align*}
 Theorem~\ref{thm:bifur-semi}
 demonstrates that the infimum $f^*$ of $f$ over $\scX$ is in the union of bifurcation sets of restricted functions $\Fr_{S}$,
 for some algebraic sets $S\subset \R^n$. 
 
 This leads us to consider the \emph{complexification} of these functions 
 $\C\!\Fr_{\C S}:\C S\longrightarrow\C$,
 where we extend the domain and range of $\FrS$ 
 to $\C S\subset \C^n$ and $\C$, respectively. 
 Subsequently, we use the inclusion of Eq.~\eqref{eq:Bf-in-Kf} 
 to show that $f^*$ lies in the union of the corresponding Rabier sets (Corollary~\ref{cor:real-complex}). 
In this way we can compute and/or approximate $f^*$.

Let $I = \set{1, \dots, r} \cup I_0$, where $I_0$ is a subset of $s + [r-s]$;
that is $I$ is a subset of $[r]$ that always contains the set $\set{1, \dots, r}$.
For any such $I$, we denote by $\VK(\bmgI)$ the common zero locus in $\K^n$ of $g_i$, for all $i\in I$, where $\K\in\{\C,\R\}$.

 \begin{definition}[Complete semialgebraic set]
 	\label{def:X-complete}
 	A semialgebraic set $\scX$, as in \eqref{eq:sa-set}, is \emph{complete} if, for every set of indices $I$, as above, the set $\VR(\bmgI)$ is a smooth complete intersection of codimension $\# I$.
 	 Then, there exists $\alpha\in\N$ and a filtration
\begin{equation}\label{eq:filtration}
\emptyset=:\scX_{-1}\subset \scX_\alpha\subset \scX_{\alpha+1}\subset \cdots \subset \scX_{n-r}:=\scX, 
\end{equation} satisfying the following properties for every $i=\alpha,\ldots,n-r$:
\begin{enumerate}[($\Pi_1$)]
	\item\label{it:filtr-dimension} 
	The set $\scX_i$ is a  basic closed semi-algebraic set with pure dimension $i$, and
	\item\label{it:filtr-intersection} 
		the Zariski closure of 
		any connected component of $\scX_{i}\setminus \scX_{i-1}$, is $\VR(\bmgI)$, where $I$ is a subset of $[s]$, such that $[r]\subset I$ and $\#I = n-i$. 
\end{enumerate} 
\end{definition}

Since $\scX$ is complete, we have $\alpha=\max(0,n-s)$.

\begin{definition}[Stratification of a complete set $\scX$]
\label{def:alg-filtr}
If $\scX$ is complete semi-algebraic subset, 
then $\cS(\scX) := \Setbar{\scX_i}{\alpha \leq i \leq n-r}$ is called an \emph{algebraic stratification} of $\scX$; we call the elements of $\cS(\scX)$ \emph{strata}.\footnote{To the best of our knowledge, in its current form, the  definition of a complete stratification does not appear in the literature.}
\end{definition}

\begin{notation}\label{not:shorthand}
 For each such $I\subset [s]$, we use the shorthand notation $\Fr_{I}$ for the restricted function
\begin{align*}
\Fr_{\VR(\bmgI)}:&~\VR(\bmgI)\longrightarrow\R.
\end{align*} 
\end{notation}

\begin{theorem}\label{thm:bifur-semi}
Let $\scX$, as in Eq.~\eqref{eq:sa-set}, be a complete semialgebraic subset of $\R^n$.
Let $F:\R^n\longrightarrow\R$ be a polynomial function, and assume that the infimum $f^*$ of the function $f:=\Fr_{\scX}:\scX\longrightarrow\R$ is not attained. Then, there exists 
a set of indices $I$, such that $[r] \subset I\subset [s]$, 
for which 
\begin{equation}\label{eq:restr}
f^* \in \cB_{\Fr_I}.
\end{equation}
\end{theorem}

\begin{proof}
Assume without loss of generality that $f^* = 0$. We will prove~\eqref{eq:restr} by finding $I$ above for which the zero betti number of $\Fr_I^{-1}(0)$ differs from that of any $\Fr_I^{-1}(z)$ with $z\in\R$ close enough to $0$. 

We start by choosing $I$ above. Since $f^*$ is not attained, and $\scX$ is closed (in the Euclidean topology), there exists $\varepsilon>0$ such that all components of $f^{-1}(]0,~\varepsilon[)$ are unbounded. We choose $\varepsilon$ small enough so that $f^{-1}(]0,~z[)$ has the same number of connected components in $\R^n$ for any $z\in~]0,~\varepsilon[$. Furthermore, one can set $\varepsilon$ even smaller if necessary so that each connected component $C$ of $f^{-1}(]0,~\varepsilon[)$ coincides with a connected component of $\sigma\cap F^{-1}(]0,~\varepsilon[)$ for some stratum $\sigma\in\cS(X)$. 

Pick one such stratum $\sigma$ above, together with a component $C$, and let $J\subset [s]$ be the index set containing $[r]$ such that $\VR(\bmg_J)$ is the Zariski closure of $\sigma$ (c.f.~\ref{it:filtr-intersection}). Since $\scX$ is complete, we can choose a sub-stratum $\delta\subset \sigma$ and a superset $I$ containing $J$ such that a connected component of $\VR(\bmgI)\cap F^{-1}(]0,~\varepsilon[) $ coincides with $\delta\cap F^{-1}(]0,~\varepsilon[)$. In what follows, we set $Z:=\VR(\bmgI)$ and $S:=\delta\cap F^{-1}(]0,~\varepsilon[)$.

Now, we use betti numbers counting of preimages under $\Fr_Z$ to finish the proof. Recall that $\cB_{\Fr_Z}$ is finite in $\R$. Then, we may take a smaller $\varepsilon>0$ if necessary so that 
\begin{equation}\label{eq:epsilon1}
]0,~\varepsilon[~\cap~ \cB_{\Fr_Z}= \emptyset
\end{equation} is satisfied. Hence, there exists $k\in\N$, such that for every $z\in~]0,~\varepsilon[$, the zero betti number satisfies
\begin{equation}\label{eq:betti01}
 b_0\big(\!\Fr_Z^{-1}(z)\big) = k.
\end{equation}
 
We finish the proof by contradiction; assume that $0\not\in\cB_{\Fr_Z}$. Then, we have $ b_0\big(\!\Fr_Z^{-1}(0)\big) = k$. Since $S$ is a connected component of $\Fr_Z^{-1}([0,~\varepsilon[)$, for each $z$ in the half-closed interval $[0,~\varepsilon[$, we get 
\[
b_0\big(\!\Fr_Z^{-1}(z)\big) = b_0\big(\!\Fr_{Z\setminus S}^{-1}(z)\big) + b_0\big(\!\Fr_S^{-1}(z)\big).
\] Then, Equality~\eqref{eq:betti01} implies that 
\[
b_0\big(\!\Fr_{Z\setminus S}^{-1}(z)\big) + b_0\big(\!\Fr_{S}^{-1}(z)\big) = b_0\big(\!\Fr_{Z\setminus S}^{-1}(0)\big) + b_0\big(\!\Fr_{S}^{-1}(0)\big).
\] Since $\Fr_{S}^{-1}(0) = \fr_{S}^{-1}(0)=\emptyset$, we get that for each $z\in~]0,~\varepsilon[$, the below equality holds:
\begin{align}\label{eq:betti-inequality}
 b_0\big(\!\Fr_{Z\setminus S}^{-1}(0)\big)> b_0\big(\!\Fr_{Z\setminus S}^{-1}(z)\big).
\end{align} Furthermore, from the definitions of $S$ and $Z$, we get that $\Fr_{Z\setminus S}^{-1}(z)$ is algebraic for each $z\in~[0,~\varepsilon[$. Therefore, inequality~\eqref{eq:betti-inequality} implies that there exists $\x\in \Fr_{Z\setminus S}^{-1}(0)$ such that $\x\in\Crit(\Fr_{Z})$. This contradicts $0\not\in \cB_{\Fr_Z}$.
\end{proof}

In what follows, for every $I\subset[s]$, we use the shorthand notation $\Fr_{\C I}$ for the complexification of the real restricted map $\Fr_{I}$. That is, 
\begin{align*}
\Fr_{\C I}:=\C\!\Fr_{\VC(\bmgI)}: &~ \VC(\bmgI)\longrightarrow\C.
\end{align*} 
\begin{corollary}\label{cor:real-complex} 
Let $\scX$, $F$, and $f$ be as in Theorem~\ref{thm:bifur-semi}. Then, there exists 
a set of indices $I$, such that $[r] \subset I\subset [s]$,
	for which it holds
\[
f^*\in \cK_{\Fr_{\C I}}.
\]
\end{corollary}
\begin{proof} Theorem~\ref{thm:bifur-semi}, together with the Rabier property~\eqref{eq:bifur-Rabier} imply that 
\[
f^*\in \cB_{\Fr_I} \subset\cK_{\Fr_{I}} ,
\] 
for some $I$ satisfying $[r]\subset I \subset [s]$. Finally, since we have $\Vert z \Vert_\C  =\Vert z \Vert_\R$ for any $z\in\R^n$, we get 
\begin{equation}\label{eq:Rabier-stratum}
\cK_{\Fr_{I}} \subset \cK_{\Fr_{\C I}} ,
\end{equation} 
for every $I\subset[s]$. This completes the proof. 
\end{proof}

\section{Constrained optimization and asymptotic critical values}
\label{sec:constrained-opt}

We provide bounds on the absolute values of points in the Rabier set of polynomial functions $f: X \longrightarrow \CC$ over smooth affine varieties $X\subset\C^n$. Whenever $f$ is given by real polynomials, this eventually leads to bounds for the optimization problem $\inf_{\x \in X\cap \R^n} f(\x)$.

\subsection{Asymptotic critical values over a smooth variety $X \subset \CC^n$}
\label{sec:acv-X}

Consider a polynomial $F\in\Q[x_1,\ldots,x_n]$ and a smooth algebraic variety $X \subseteq \CC^n$ of dimension $\delta = n-r$,
that is the zero set of the polynomials $\g = \{ g_1, \dots, g_r \} \in \QQ[\x]$. 
As our goal is to bound the asymptotic critical values of the function $f:=\Fr_X$, it is without loss of generality to consider $X$ to be complete intersection; see Theorem~\ref{thm:make-square} and Remark~\ref{rem:make-square}.
If we have more than $r$ polynomials, then we can consider $r$ generic linear combinations of them to construct a "complete intersection" system,
at the expense of adding some additional component(s) in $X$. 
The original components of the initial system can be identified by considering several generic linear combinations if necessary. As this technique does not change a lot the bitsize of the polynomial to compute with, and we only interested to bound (and not compute) the related quantities, we do not detail further. 
We refer the reader to \cite{det-chow-24} for additional details 
and an application in computing the Chow form of a variety.

%
%

Consider the $(r+1)\times n$ matrix 
\[
	C = 
	\left(
	\begin{array}{cc}
		\nabla F \\
		\Jac(\g)
	\end{array}
	\right) ,
\]
where $\nabla F$ is the gradient of $F$ and has dimension $1 \times n$, 
while $\Jac(\g)$ is the Jacobian matrix of $\g$ and has dimension $r \times n$.
Consider the set of all subsets of $[n]$ of cardinality $r+1$;
we denoted by $\left[\begin{smallmatrix}n \\r+1\end{smallmatrix}\right]$.
Then, $I\in \left[\begin{smallmatrix}n \\r+1\end{smallmatrix}\right]$ $\Leftrightarrow$ $I \subseteq [n]$ and  $\abs{I} = r+1$. There are $s = {n \choose r+1}$ such subsets $I$ and to each of them, we associate an integer $i$, where $1 \leq i \leq s$.

Let $M_{i}$ be the $(r+1) \times (r+1)$ square submatrix of $C$, obtained by selecting the columns of $C$ with indices in $I$.
For $j \in I$, let $M_{i, j}$ be the square $r \times r$ submatrix of $M_{i}$, obtained by omitting the first row (which corresponds to $\nabla F$) and the $j$-th column; there are $r+1$ such matrices, for every $i$.

For a specific $i$ and $j \in I$, let 
$m_{i}$ and $m_{i,j}$ be the determinants of $M_{i}$ and $M_{i,j}$, respectively;
they are polynomials in $\ZZ[\x]$. 
We also need the definition of the rational function 
\[
w_{i, j} = \frac{m_{i}}{m_{i, j}} = \frac{\det M_i}{\det M_{i,j}} \enspace.
\]

Consider the vector $\bj = (j_1, \dots, j_s) \in \NN^s$, where $j_i \in I$.
There are $(r+1)^s$ such vectors.
For each $\bj$, we consider the  map $\Phi_{\bj}$, which is 
\begin{equation}
\label{eq:Phi-j}
\begin{array}{cccclll}
	\Phi_{\bj} &:& X &\longrightarrow& \CC^n \times \CC^{(n+1)s} \\\
	&& \x &\longmapsto& ( F(\x), \{ w_{i, j_i}(\x), x_1 w_{i, j_i}(\x), \dots, x_n w_{i, j_i}(\x) \}_{i \in [s]}) .
\end{array}
 \end{equation}

If $\cG_{\bj} := \overline{\image(\Phi_{\bj})}$,
and $\cG = \bigcup_{\bj} \cG_{\bj}$,
then $\cK(f) = L \cap \Gamma$, where $L = \underbrace{(0, \dots, 0)}_{(n+1)s \text{ times}} \times \CC$ (see~\cite{jelonek2005quantitative}). 

Upper and lower bounds on the roots of these polynomials, will also hold truer for the asymptotic critical values.

\subsubsection{Resultant systems using determinants}
\label{sec:res-sys}

Our presentation is based on van der Waerden~\cite{vdWaerden-alebra-37} and Yap~\cite{Yap-algebra-book}.
Let $\bm{A} = \{A_1, \dots, A_p\} \subseteq (\ZZ[\bm{c}])[\x]$ be a system of polynomial equations in the variables $\x$,
the coefficients of which are polynomials with integer coefficients in the additional set of variables $\bm{c}$. 
Let the degree of each polynomial $A_i$ with respect to $\x$
 be $d_i$
and $d = \max_{i \in [p]} d_i$.

The resultant system is a set of polynomials in the coefficients $\bm{c}$ having the property that all of them vanish if 
the polynomials in $\bm{A}$ have a common non-trivial solution.
In the case where $p=n$, then the resultant system consists of a single polynomial, which we call the (homogeneous) resultant of the system.

Let $\nu_m = {m + n - 1 \choose n-1} \leq m^{n-1}$, where $m$ is a positive integer that we will specify in the sequel.
Let $\mathtt{P}^m = \{ \x^{\balpha} \,|\, \abs{\balpha} = m \}$,
that is the set of all monomials in $\x$ of degree $m$.
Next, we consider the set 
\[
	\bm{A}_m = \Setbar{ \x^{\balpha}A}{\deg(\x^{\balpha}A) = m, A \in \bm{A}, \x^{\balpha} \in \mathtt{P}^m } \enspace.
\]
It holds $\abs{\bm{A}_m} \geq \nu_m$.
Let $M_m$ be the $\left(\abs{\bm{A}_m} \times \nu_m\right)$-matrix,
the rows of which correspond to the polynomials in $\bm{A}_m$
and its columns correspond to the elements of $\mathtt{P}^m$.

Let $R_m \subseteq \ZZ[\bm{c}]$ be the set of all the $\left(\nu_m \times \nu_m\right)$-minors of the matrix $M_m$.

\begin{theorem}[{\cite[XI~Thm.~9~\&~13]{Yap-algebra-book}}]
	If $m \geq 1 + n(13~d^n - 1)$, then $R_m$ is a resultant system for $\bm{A}$. 
\end{theorem}

The $13~d^n$ factor in the bound for $m$ corresponds to an effective bound on the Nullstellensatz. We note that we can also use improved bounds of effective Nullstellensatz, but this bound suffices for our purposes. 
For our purposes, we do not need to count or manipulate with the polynomials in $R_m$. We just need to bound their degree, as polynomials in $\bm{c}$, and their bitsize.

\subsubsection{Bounding the asymptotic critical values using resultant systems}
\label{sec:acv-X-res-sys}


To obtain worst case bound on the asymptotic critical values, it suffices to consider worst case bounds for $\cG_{\bj}$, for a fixed pair $(\bj,I)$, with $I\in\left[\begin{smallmatrix}n \\r+1\end{smallmatrix}\right]$ and  $\bj = (j_1, \dots, j_s) \in \NN^s$ satisfying $j_i \in I$.
We emphasize that our goal is \emph{not} to compute these values, but rather to bound them effectively.

Let $X_{\bj} = \VV( \set{m_{i,j_i}}_{i \in [s]}) \subseteq \CC^n$
be the zero set of all denominators in $\Phi_{\bj}$, Eq.~\eqref{eq:Phi-j}.
Let $\delta_{i} = \dim(X_{\bj})$ be the dimension $X_{\bj}$. 
As it is not a complete intersection, we can consider $n - \delta_i$ generic linear combinations of the polynomials $m_{i,j_i}$.
Let these new polynomials be $\set{h_1, \dots, h_{n-\delta_i}}$. 

Consider the set of polynomials 
\begin{multline}
	J = \Big\{
	\set{g_i(\x)}_{i \in [r]}, F(\x) - z,  \\
	\set{m_{i,j_i}(\x) y_i - m_i(\x), m_{i,j_i}(\x) y_i - x_1 m_i(\x),
	\dots,  m_{i,j_i}(\x) y_i - x_n m_i(\x) }_{i\in[s]}, \\	
	t \, h(\x) -1 
	\Big\}
	\subseteq \ZZ[\x, t, y_1, \dots, y_s, z] ,
\end{multline}
where $h(\x) = \prod_{i \in [n - \delta_i]} h_i(\x)$.

This set contains $r + 1 + (n+1){n \choose r+1} + 1$ equations.
We should eliminate the variable $(\x, t)$; this would result in polynomials in $\ZZ[y_1, \dots, y_s, z]$. Then, if we set $y_i = 0$, we obtain univariate polynomials in $z$, the roots of which contains the asymptotic critical values of $f$.  

Let $F$ be of size $(d, \tau)$ and $g_i$ of size $(d_1, \tau)$.
Standard calculations based on Claim~\ref{claim:mul-mpoly}
result that 
$m_{i, j}$ is of size $(\OO(r d_1), \sOO(r(d + n + \tau)))$,
$m_i$ is of size $(\OO(d + rd_1), \sOO(r(d + n + \tau) + d_1))$.

Based on the results in the appendix, the polynomials $h_i$ 
have the same size as the polynomials $m_{i,j}$. 
Thus, $h(\x)$ is of size $(\OO(n r d_1), \sOO(n r(d + n + \tau)))$,

To eliminate the variables $\x$ from the polynomials in $J$,
we will use resultant systems from Sec.~\ref{sec:res-sys}.
The maximum degree of the involved polynomials is $\OO(n r d_1)$,
thus the Nullstellensatz bound becomes $\OO((n r d_1)^n)$
which in turn implies that the various matrices 
are of dimension $\OO((n r d_1)^{n^2})$.

The elements of the matrix (or matrices) are polynomials 
in the variables $y_1, \dots, y_s, z$. 
Their degree with respect to $y_i$ and $z$ is one.
Their maximum bitsize is $\sOO(r(d + n + \tau) + d_1))$.
Hence, their determinant is a polynomial in $(\ZZ[y_1, \dots, y_s])[z]$ of degree $\OO((n r d_1)^{n^2})$.
and bitsize 
\[
\eta = \sOO(r(d + n + \tau) + d_1)(n r d_1)^{n^2} ) .
\] 
If we set the variables $y_i$ to zero, then we obtain a univariate polynomial in $\ZZ[z]$ of the same size, the roots of which contain the asymptotic critical values. 

The arguments above yield the following result.

\begin{theorem}
	Let $F \in \ZZ[\x]$ be of degree $d$ and bitsize $\tau$.
	Let $X = \VV(g_1, \dots, g_r) \cap \RR^n$ be a smooth real algebraic variety and $g_i$ be of size $(d_1, \tau)$.
	The optimum value of the problem $f^* = \inf_{\x \in X} F(\x)$ 
	is an algebraic number of degree $\OO((n r d_1)^{n^2})$
	such that 
	$2^{-\eta} \leq \abs{f^*} \leq 2^{\eta}$, where 
$\eta = \sOO(r(d + n + \tau) + d_1)(n r d_1)^{n^2} )$. 

\end{theorem}

Finally, thanks to Corollary~\ref{cor:real-complex}, we deduce our min Theorem~\ref{thm:main}.

\section{The Newton non-degenerate case}\label{sec:Newton-non}
As Theorem~\ref{thm:bifur-semi} presents, in order to locate the infimum of a semi-algebraic function, one can consider a polynomial function instead, and compute its bifurcation set. The latter, however, is intractable for arbitrary functions. Accordingly, we can use the Rabier set to effectively compute a superset containing it. Albeit the computation itself has high complexity as we will see in~\S\ref{sec:constrained-opt}. 

In this section, we introduce a large family of polynomial functions, and show that one can effectively approximate their bifurcation set using less expensive methods than the Rabier set.

\subsection{Preliminaries on polytopes}
\label{sub:prelim} 
A subset $\Pi\subset\R^n$ is called a \emph{polyhedron} if it is the intersection of finitely-many closed half-spaces. The boundary of one such half-space is called a \emph{supporting hyperplane} of $\Pi$. We say that a set $\Phi$ is a \emph{face} of $\Pi$,
we indicate this using the notation $\Phi\prec \Pi$, if it is the intersection of a supporting hyperplane $H$ of $\Pi$ with its boundary, i.e., $\Phi=H\cap\partial \Pi$. We say that $\Phi$ is  \emph{origin} if $\Phi$ contains the point $\unze:=(0,\ldots,0)\in\R^n$.
A \emph{polytope} is a bounded polyhedron.

A \emph{tuple of polytopes}, or a \emph{tuple} for short, is a map $\bdel$ from a finite set $K\subset \N$ to the set of polytopes in $\R^n$. We call $K$ the \emph{support} of $\bdel$ and we denote it by $[\bdel]$. The Minkowski sum of elements in $\bdel$
\[
	\left\{ \sum\nolimits_{k\in K}\bma_k \,|\, \bma_k \in \Delta_k  \right\},
\]
 is also a polytope; we denote it by $\sum \bdel$. 
 In our setting, the \emph{Minkowski sum} of any two subsets $X,Y\subset\R^n$ is the coordinate-wise sum $
X+Y:=\{x+y~|~x\in X,~y\in Y\}$. 
The \emph{dimension} of $\bdel$ is defined as
\begin{align}\label{eq:dimension}
\dim \bdel:= & \dim (\sum \bdel)- \# [\bdel].
\end{align}

For any $i\in [\bdel]$, 
the $i$-th polytope in the tuple is $\Delta_i$. 
For any $I\subset [\bdel]$,
$\bdel_I:= \Setbar{\Delta_i}{i \in I}$ is the \emph{sub-tuple} (of $\bdel$ associated to $I$). 
 A tuple $\bgam$ is a \emph{tuple-face} of $\bdel$ (or, simply, \emph{face} whenever it is clear from the context) if $[\bgam] = [\bdel]$,
 $\Gamma_i\prec\Delta_i$, for each $i\in[\bgam]$, 
 and $\sum\bgam\prec\sum \bdel$. 
  
 We say that $\bgam$ is a \emph{facing} of $\bdel$,
 if there is a face $\bgam'$ of $ \bdel$ 
 such that $\bgam=\bgam'_I$ and $I = [\bgam]$.
 We also use the notation $\bgam\prec \bdel$ for the facings.
%

\begin{definition}
	\label{def:important_facing}
	A facing $\bgam\prec\bdel$ is said to be \emph{important}, if there exists a face $\bgam'\prec\bdel$  such that  $\bgam = \bgam'_{[\bgam]}$ and 
\[
	\dim \bgam\leq \dim\bgam'_I,\quad\forall~I\supset [\bgam].
\]
\end{definition}

\begin{figure}[htb]

\tikzset{every picture/.style={line width=0.75pt}} 

\begin{tikzpicture}[x=0.75pt,y=0.75pt,yscale=-1,xscale=1]

\draw  [fill={rgb, 255:red, 0; green, 0; blue, 0 }  ,fill opacity=0.4 ] (273.83,240.5) -- (251.67,240.5) -- (251.67,218.71) -- cycle ;
\draw [color={rgb, 255:red, 155; green, 155; blue, 155 }  ,draw opacity=0.4 ] [dash pattern={on 3pt off 1.5pt}]  (272.99,70.34) -- (285.01,63.82) ;
\draw [shift={(286.77,62.87)}, rotate = 151.53] [fill={rgb, 255:red, 155; green, 155; blue, 155 }  ,fill opacity=0.4 ][line width=0.08]  [draw opacity=0] (4.8,-1.2) -- (0,0) -- (4.8,1.2) -- cycle    ;
\draw [color={rgb, 255:red, 155; green, 155; blue, 155 }  ,draw opacity=0.4 ] [dash pattern={on 3pt off 1.5pt}]  (161.83,81.5) -- (215.43,53.2) ;
\draw [shift={(217.2,52.27)}, rotate = 152.17] [fill={rgb, 255:red, 155; green, 155; blue, 155 }  ,fill opacity=0.4 ][line width=0.08]  [draw opacity=0] (4.8,-1.2) -- (0,0) -- (4.8,1.2) -- cycle    ;
\draw    (161.83,81.5) -- (202.75,59.98) ;
\draw [color={rgb, 255:red, 155; green, 155; blue, 155 }  ,draw opacity=0.4 ] [dash pattern={on 3pt off 1.5pt}]  (71.83,81.5) -- (129.83,81.5) ;
\draw [shift={(131.83,81.5)}, rotate = 180] [fill={rgb, 255:red, 155; green, 155; blue, 155 }  ,fill opacity=0.4 ][line width=0.08]  [draw opacity=0] (4.8,-1.2) -- (0,0) -- (4.8,1.2) -- cycle    ;
\draw [color={rgb, 255:red, 155; green, 155; blue, 155 }  ,draw opacity=0.4 ] [dash pattern={on 3pt off 1.5pt}]  (71.67,81.5) -- (71.67,43.5) ;
\draw [shift={(71.67,41.5)}, rotate = 90] [fill={rgb, 255:red, 155; green, 155; blue, 155 }  ,fill opacity=0.4 ][line width=0.08]  [draw opacity=0] (4.8,-1.2) -- (0,0) -- (4.8,1.2) -- cycle    ;
\draw  [color={rgb, 255:red, 0; green, 0; blue, 0 }  ,draw opacity=1 ][fill={rgb, 255:red, 0; green, 0; blue, 0 }  ,fill opacity=1 ] (92.69,81.5) .. controls (92.69,80.87) and (93.2,80.35) .. (93.83,80.35) .. controls (94.47,80.35) and (94.98,80.87) .. (94.98,81.5) .. controls (94.98,82.13) and (94.47,82.65) .. (93.83,82.65) .. controls (93.2,82.65) and (92.69,82.13) .. (92.69,81.5) -- cycle ;
\draw  [color={rgb, 255:red, 0; green, 0; blue, 0 }  ,draw opacity=1 ][fill={rgb, 255:red, 0; green, 0; blue, 0 }  ,fill opacity=1 ] (114.69,81.5) .. controls (114.69,80.87) and (115.2,80.35) .. (115.83,80.35) .. controls (116.47,80.35) and (116.98,80.87) .. (116.98,81.5) .. controls (116.98,82.13) and (116.47,82.65) .. (115.83,82.65) .. controls (115.2,82.65) and (114.69,82.13) .. (114.69,81.5) -- cycle ;
\draw [color={rgb, 255:red, 155; green, 155; blue, 155 }  ,draw opacity=0.4 ] [dash pattern={on 3pt off 1.5pt}]  (161.83,81.5) -- (199.83,81.5) ;
\draw [shift={(201.83,81.5)}, rotate = 180] [fill={rgb, 255:red, 155; green, 155; blue, 155 }  ,fill opacity=0.4 ][line width=0.08]  [draw opacity=0] (4.8,-1.2) -- (0,0) -- (4.8,1.2) -- cycle    ;
\draw [color={rgb, 255:red, 155; green, 155; blue, 155 }  ,draw opacity=0.4 ] [dash pattern={on 3pt off 1.5pt}]  (161.67,81.5) -- (161.67,43.5) ;
\draw [shift={(161.67,41.5)}, rotate = 90] [fill={rgb, 255:red, 155; green, 155; blue, 155 }  ,fill opacity=0.4 ][line width=0.08]  [draw opacity=0] (4.8,-1.2) -- (0,0) -- (4.8,1.2) -- cycle    ;
\draw  [color={rgb, 255:red, 0; green, 0; blue, 0 }  ,draw opacity=1 ][fill={rgb, 255:red, 0; green, 0; blue, 0 }  ,fill opacity=1 ] (181.15,70.74) .. controls (181.15,70.11) and (181.66,69.59) .. (182.29,69.59) .. controls (182.92,69.59) and (183.44,70.11) .. (183.44,70.74) .. controls (183.44,71.37) and (182.92,71.89) .. (182.29,71.89) .. controls (181.66,71.89) and (181.15,71.37) .. (181.15,70.74) -- cycle ;
\draw  [color={rgb, 255:red, 0; green, 0; blue, 0 }  ,draw opacity=1 ][fill={rgb, 255:red, 0; green, 0; blue, 0 }  ,fill opacity=1 ] (201.6,59.98) .. controls (201.6,59.35) and (202.12,58.83) .. (202.75,58.83) .. controls (203.38,58.83) and (203.9,59.35) .. (203.9,59.98) .. controls (203.9,60.61) and (203.38,61.12) .. (202.75,61.12) .. controls (202.12,61.12) and (201.6,60.61) .. (201.6,59.98) -- cycle ;
\draw  [dash pattern={on 3.75pt off 1.5pt}]  (251.83,81.5) -- (272.99,70.34) ;
\draw [color={rgb, 255:red, 155; green, 155; blue, 155 }  ,draw opacity=0.4 ] [dash pattern={on 3pt off 1.5pt}]  (251.83,81.5) -- (289.83,81.5) ;
\draw [shift={(291.83,81.5)}, rotate = 180] [fill={rgb, 255:red, 155; green, 155; blue, 155 }  ,fill opacity=0.4 ][line width=0.08]  [draw opacity=0] (4.8,-1.2) -- (0,0) -- (4.8,1.2) -- cycle    ;
\draw [color={rgb, 255:red, 155; green, 155; blue, 155 }  ,draw opacity=0.4 ] [dash pattern={on 3pt off 1.5pt}]  (251.67,81.5) -- (251.67,43.5) ;
\draw [shift={(251.67,41.5)}, rotate = 90] [fill={rgb, 255:red, 155; green, 155; blue, 155 }  ,fill opacity=0.4 ][line width=0.08]  [draw opacity=0] (4.8,-1.2) -- (0,0) -- (4.8,1.2) -- cycle    ;
\draw  [color={rgb, 255:red, 0; green, 0; blue, 0 }  ,draw opacity=1 ][fill={rgb, 255:red, 0; green, 0; blue, 0 }  ,fill opacity=1 ] (272.69,81.5) .. controls (272.69,80.87) and (273.2,80.35) .. (273.83,80.35) .. controls (274.47,80.35) and (274.98,80.87) .. (274.98,81.5) .. controls (274.98,82.13) and (274.47,82.65) .. (273.83,82.65) .. controls (273.2,82.65) and (272.69,82.13) .. (272.69,81.5) -- cycle ;
\draw  [color={rgb, 255:red, 0; green, 0; blue, 0 }  ,draw opacity=1 ][fill={rgb, 255:red, 0; green, 0; blue, 0 }  ,fill opacity=1 ] (250.69,59.71) .. controls (250.69,59.08) and (251.2,58.56) .. (251.83,58.56) .. controls (252.47,58.56) and (252.98,59.08) .. (252.98,59.71) .. controls (252.98,60.34) and (252.47,60.85) .. (251.83,60.85) .. controls (251.2,60.85) and (250.69,60.34) .. (250.69,59.71) -- cycle ;
\draw  [color={rgb, 255:red, 0; green, 0; blue, 0 }  ,draw opacity=1 ][fill={rgb, 255:red, 0; green, 0; blue, 0 }  ,fill opacity=1 ] (271.85,70.34) .. controls (271.85,69.71) and (272.36,69.19) .. (272.99,69.19) .. controls (273.63,69.19) and (274.14,69.71) .. (274.14,70.34) .. controls (274.14,70.97) and (273.63,71.48) .. (272.99,71.48) .. controls (272.36,71.48) and (271.85,70.97) .. (271.85,70.34) -- cycle ;
\draw    (251.83,81.5) -- (251.83,59.71) ;
\draw    (251.83,81.5) -- (273.83,81.5) ;
\draw    (272.99,70.34) -- (251.83,59.71) ;
\draw    (273.83,81.5) -- (272.99,70.34) ;
\draw [color={rgb, 255:red, 155; green, 155; blue, 155 }  ,draw opacity=0.4 ] [dash pattern={on 3pt off 1.5pt}]  (71.67,81.5) -- (125.26,53.2) ;
\draw [shift={(127.03,52.27)}, rotate = 152.17] [fill={rgb, 255:red, 155; green, 155; blue, 155 }  ,fill opacity=0.4 ][line width=0.08]  [draw opacity=0] (4.8,-1.2) -- (0,0) -- (4.8,1.2) -- cycle    ;
\draw    (71.83,81.5) -- (115.83,81.5) ;
\draw  [color={rgb, 255:red, 0; green, 0; blue, 0 }  ,draw opacity=1 ][fill={rgb, 255:red, 255; green, 255; blue, 255 }  ,fill opacity=1 ] (69.02,81.5) .. controls (69.02,80.04) and (70.2,78.85) .. (71.67,78.85) .. controls (73.13,78.85) and (74.32,80.04) .. (74.32,81.5) .. controls (74.32,82.96) and (73.13,84.15) .. (71.67,84.15) .. controls (70.2,84.15) and (69.02,82.96) .. (69.02,81.5) -- cycle ;
\draw  [color={rgb, 255:red, 0; green, 0; blue, 0 }  ,draw opacity=1 ][fill={rgb, 255:red, 255; green, 255; blue, 255 }  ,fill opacity=1 ] (249.02,81.5) .. controls (249.02,80.04) and (250.2,78.85) .. (251.67,78.85) .. controls (253.13,78.85) and (254.32,80.04) .. (254.32,81.5) .. controls (254.32,82.96) and (253.13,84.15) .. (251.67,84.15) .. controls (250.2,84.15) and (249.02,82.96) .. (249.02,81.5) -- cycle ;
\draw  [color={rgb, 255:red, 255; green, 255; blue, 255 }  ,draw opacity=1 ][fill={rgb, 255:red, 255; green, 255; blue, 255 }  ,fill opacity=1 ] (263.85,73.85) .. controls (263.85,72.72) and (264.77,71.81) .. (265.9,71.81) .. controls (267.03,71.81) and (267.95,72.72) .. (267.95,73.85) .. controls (267.95,74.99) and (267.03,75.9) .. (265.9,75.9) .. controls (264.77,75.9) and (263.85,74.99) .. (263.85,73.85) -- cycle ;
\draw    (273.83,81.5) -- (267.68,75.41) -- (251.83,59.71) ;
\draw  [color={rgb, 255:red, 0; green, 0; blue, 0 }  ,draw opacity=1 ][fill={rgb, 255:red, 0; green, 0; blue, 0 }  ,fill opacity=1 ] (92.69,172.5) .. controls (92.69,171.87) and (93.2,171.35) .. (93.83,171.35) .. controls (94.47,171.35) and (94.98,171.87) .. (94.98,172.5) .. controls (94.98,173.13) and (94.47,173.65) .. (93.83,173.65) .. controls (93.2,173.65) and (92.69,173.13) .. (92.69,172.5) -- cycle ;
\draw  [color={rgb, 255:red, 0; green, 0; blue, 0 }  ,draw opacity=1 ][fill={rgb, 255:red, 0; green, 0; blue, 0 }  ,fill opacity=1 ] (114.69,172.5) .. controls (114.69,171.87) and (115.2,171.35) .. (115.83,171.35) .. controls (116.47,171.35) and (116.98,171.87) .. (116.98,172.5) .. controls (116.98,173.13) and (116.47,173.65) .. (115.83,173.65) .. controls (115.2,173.65) and (114.69,173.13) .. (114.69,172.5) -- cycle ;
\draw    (71.83,172.5) -- (115.83,172.5) ;
\draw    (161.83,81.5) -- (183.83,81.5) ;
\draw  [color={rgb, 255:red, 0; green, 0; blue, 0 }  ,draw opacity=1 ][fill={rgb, 255:red, 0; green, 0; blue, 0 }  ,fill opacity=1 ] (182.69,81.5) .. controls (182.69,80.87) and (183.2,80.35) .. (183.83,80.35) .. controls (184.47,80.35) and (184.98,80.87) .. (184.98,81.5) .. controls (184.98,82.13) and (184.47,82.65) .. (183.83,82.65) .. controls (183.2,82.65) and (182.69,82.13) .. (182.69,81.5) -- cycle ;
\draw    (183.83,81.5) -- (202.75,59.98) ;
\draw  [color={rgb, 255:red, 0; green, 0; blue, 0 }  ,draw opacity=1 ][fill={rgb, 255:red, 255; green, 255; blue, 255 }  ,fill opacity=1 ] (159.02,81.5) .. controls (159.02,80.04) and (160.2,78.85) .. (161.67,78.85) .. controls (163.13,78.85) and (164.32,80.04) .. (164.32,81.5) .. controls (164.32,82.96) and (163.13,84.15) .. (161.67,84.15) .. controls (160.2,84.15) and (159.02,82.96) .. (159.02,81.5) -- cycle ;
\draw  [color={rgb, 255:red, 0; green, 0; blue, 0 }  ,draw opacity=1 ][fill={rgb, 255:red, 0; green, 0; blue, 0 }  ,fill opacity=1 ] (70.69,172.5) .. controls (70.69,171.87) and (71.2,171.35) .. (71.83,171.35) .. controls (72.47,171.35) and (72.98,171.87) .. (72.98,172.5) .. controls (72.98,173.13) and (72.47,173.65) .. (71.83,173.65) .. controls (71.2,173.65) and (70.69,173.13) .. (70.69,172.5) -- cycle ;
\draw  [color={rgb, 255:red, 0; green, 0; blue, 0 }  ,draw opacity=1 ][fill={rgb, 255:red, 0; green, 0; blue, 0 }  ,fill opacity=1 ] (181.15,161.92) .. controls (181.15,161.29) and (181.66,160.78) .. (182.29,160.78) .. controls (182.92,160.78) and (183.44,161.29) .. (183.44,161.92) .. controls (183.44,162.55) and (182.92,163.07) .. (182.29,163.07) .. controls (181.66,163.07) and (181.15,162.55) .. (181.15,161.92) -- cycle ;
\draw  [color={rgb, 255:red, 0; green, 0; blue, 0 }  ,draw opacity=1 ][fill={rgb, 255:red, 0; green, 0; blue, 0 }  ,fill opacity=1 ] (201.6,151.16) .. controls (201.6,150.53) and (202.12,150.02) .. (202.75,150.02) .. controls (203.38,150.02) and (203.9,150.53) .. (203.9,151.16) .. controls (203.9,151.79) and (203.38,152.31) .. (202.75,152.31) .. controls (202.12,152.31) and (201.6,151.79) .. (201.6,151.16) -- cycle ;
\draw  [color={rgb, 255:red, 0; green, 0; blue, 0 }  ,draw opacity=1 ][fill={rgb, 255:red, 0; green, 0; blue, 0 }  ,fill opacity=1 ] (182.69,172.68) .. controls (182.69,172.05) and (183.2,171.54) .. (183.83,171.54) .. controls (184.47,171.54) and (184.98,172.05) .. (184.98,172.68) .. controls (184.98,173.31) and (184.47,173.83) .. (183.83,173.83) .. controls (183.2,173.83) and (182.69,173.31) .. (182.69,172.68) -- cycle ;
\draw  [color={rgb, 255:red, 0; green, 0; blue, 0 }  ,draw opacity=1 ][fill={rgb, 255:red, 0; green, 0; blue, 0 }  ,fill opacity=1 ] (160.69,172.68) .. controls (160.69,172.05) and (161.2,171.54) .. (161.83,171.54) .. controls (162.47,171.54) and (162.98,172.05) .. (162.98,172.68) .. controls (162.98,173.31) and (162.47,173.83) .. (161.83,173.83) .. controls (161.2,173.83) and (160.69,173.31) .. (160.69,172.68) -- cycle ;
\draw  [fill={rgb, 255:red, 0; green, 0; blue, 0 }  ,fill opacity=0.4 ] (202.75,151.16) -- (183.83,172.68) -- (161.83,172.68) -- cycle ;
\draw  [color={rgb, 255:red, 0; green, 0; blue, 0 }  ,draw opacity=1 ][fill={rgb, 255:red, 0; green, 0; blue, 0 }  ,fill opacity=1 ] (250.69,172.5) .. controls (250.69,171.87) and (251.2,171.35) .. (251.83,171.35) .. controls (252.47,171.35) and (252.98,171.87) .. (252.98,172.5) .. controls (252.98,173.13) and (252.47,173.65) .. (251.83,173.65) .. controls (251.2,173.65) and (250.69,173.13) .. (250.69,172.5) -- cycle ;
\draw  [color={rgb, 255:red, 0; green, 0; blue, 0 }  ,draw opacity=1 ][fill={rgb, 255:red, 0; green, 0; blue, 0 }  ,fill opacity=1 ] (272.69,172.5) .. controls (272.69,171.87) and (273.2,171.35) .. (273.83,171.35) .. controls (274.47,171.35) and (274.98,171.87) .. (274.98,172.5) .. controls (274.98,173.13) and (274.47,173.65) .. (273.83,173.65) .. controls (273.2,173.65) and (272.69,173.13) .. (272.69,172.5) -- cycle ;
\draw  [color={rgb, 255:red, 0; green, 0; blue, 0 }  ,draw opacity=1 ][fill={rgb, 255:red, 0; green, 0; blue, 0 }  ,fill opacity=1 ] (279.52,157.21) .. controls (279.52,156.58) and (280.03,156.06) .. (280.67,156.06) .. controls (281.3,156.06) and (281.81,156.58) .. (281.81,157.21) .. controls (281.81,157.84) and (281.3,158.35) .. (280.67,158.35) .. controls (280.03,158.35) and (279.52,157.84) .. (279.52,157.21) -- cycle ;
\draw  [fill={rgb, 255:red, 0; green, 0; blue, 0 }  ,fill opacity=0.4 ] (280.67,157.21) -- (273.83,172.5) -- (251.83,172.5) -- cycle ;
\draw  [color={rgb, 255:red, 0; green, 0; blue, 0 }  ,draw opacity=1 ][fill={rgb, 255:red, 0; green, 0; blue, 0 }  ,fill opacity=1 ] (92.69,240.5) .. controls (92.69,239.87) and (93.2,239.35) .. (93.83,239.35) .. controls (94.47,239.35) and (94.98,239.87) .. (94.98,240.5) .. controls (94.98,241.13) and (94.47,241.65) .. (93.83,241.65) .. controls (93.2,241.65) and (92.69,241.13) .. (92.69,240.5) -- cycle ;
\draw  [color={rgb, 255:red, 0; green, 0; blue, 0 }  ,draw opacity=1 ][fill={rgb, 255:red, 0; green, 0; blue, 0 }  ,fill opacity=1 ] (114.69,240.5) .. controls (114.69,239.87) and (115.2,239.35) .. (115.83,239.35) .. controls (116.47,239.35) and (116.98,239.87) .. (116.98,240.5) .. controls (116.98,241.13) and (116.47,241.65) .. (115.83,241.65) .. controls (115.2,241.65) and (114.69,241.13) .. (114.69,240.5) -- cycle ;
\draw  [color={rgb, 255:red, 0; green, 0; blue, 0 }  ,draw opacity=1 ][fill={rgb, 255:red, 0; green, 0; blue, 0 }  ,fill opacity=1 ] (272.69,240.5) .. controls (272.69,239.87) and (273.2,239.35) .. (273.83,239.35) .. controls (274.47,239.35) and (274.98,239.87) .. (274.98,240.5) .. controls (274.98,241.13) and (274.47,241.65) .. (273.83,241.65) .. controls (273.2,241.65) and (272.69,241.13) .. (272.69,240.5) -- cycle ;
\draw  [color={rgb, 255:red, 0; green, 0; blue, 0 }  ,draw opacity=1 ][fill={rgb, 255:red, 0; green, 0; blue, 0 }  ,fill opacity=1 ] (250.52,218.71) .. controls (250.52,218.08) and (251.03,217.56) .. (251.67,217.56) .. controls (252.3,217.56) and (252.81,218.08) .. (252.81,218.71) .. controls (252.81,219.34) and (252.3,219.85) .. (251.67,219.85) .. controls (251.03,219.85) and (250.52,219.34) .. (250.52,218.71) -- cycle ;
\draw    (71.83,240.5) -- (115.83,240.5) ;
\draw    (161.83,240.5) -- (183.83,240.5) ;
\draw  [color={rgb, 255:red, 0; green, 0; blue, 0 }  ,draw opacity=1 ][fill={rgb, 255:red, 0; green, 0; blue, 0 }  ,fill opacity=1 ] (182.69,240.5) .. controls (182.69,239.87) and (183.2,239.35) .. (183.83,239.35) .. controls (184.47,239.35) and (184.98,239.87) .. (184.98,240.5) .. controls (184.98,241.13) and (184.47,241.65) .. (183.83,241.65) .. controls (183.2,241.65) and (182.69,241.13) .. (182.69,240.5) -- cycle ;
\draw  [color={rgb, 255:red, 0; green, 0; blue, 0 }  ,draw opacity=1 ][fill={rgb, 255:red, 0; green, 0; blue, 0 }  ,fill opacity=1 ] (250.52,240.5) .. controls (250.52,239.87) and (251.03,239.35) .. (251.67,239.35) .. controls (252.3,239.35) and (252.81,239.87) .. (252.81,240.5) .. controls (252.81,241.13) and (252.3,241.65) .. (251.67,241.65) .. controls (251.03,241.65) and (250.52,241.13) .. (250.52,240.5) -- cycle ;
\draw  [color={rgb, 255:red, 0; green, 0; blue, 0 }  ,draw opacity=1 ][fill={rgb, 255:red, 0; green, 0; blue, 0 }  ,fill opacity=1 ] (160.69,240.5) .. controls (160.69,239.87) and (161.2,239.35) .. (161.83,239.35) .. controls (162.47,239.35) and (162.98,239.87) .. (162.98,240.5) .. controls (162.98,241.13) and (162.47,241.65) .. (161.83,241.65) .. controls (161.2,241.65) and (160.69,241.13) .. (160.69,240.5) -- cycle ;
\draw  [color={rgb, 255:red, 0; green, 0; blue, 0 }  ,draw opacity=1 ][fill={rgb, 255:red, 0; green, 0; blue, 0 }  ,fill opacity=1 ] (70.69,240.5) .. controls (70.69,239.87) and (71.2,239.35) .. (71.83,239.35) .. controls (72.47,239.35) and (72.98,239.87) .. (72.98,240.5) .. controls (72.98,241.13) and (72.47,241.65) .. (71.83,241.65) .. controls (71.2,241.65) and (70.69,241.13) .. (70.69,240.5) -- cycle ;
\draw    (479.83,172.5) -- (520.75,150.98) ;
\draw  [color={rgb, 255:red, 0; green, 0; blue, 0 }  ,draw opacity=1 ][fill={rgb, 255:red, 0; green, 0; blue, 0 }  ,fill opacity=1 ] (499.15,161.74) .. controls (499.15,161.11) and (499.66,160.59) .. (500.29,160.59) .. controls (500.92,160.59) and (501.44,161.11) .. (501.44,161.74) .. controls (501.44,162.37) and (500.92,162.89) .. (500.29,162.89) .. controls (499.66,162.89) and (499.15,162.37) .. (499.15,161.74) -- cycle ;
\draw  [color={rgb, 255:red, 0; green, 0; blue, 0 }  ,draw opacity=1 ][fill={rgb, 255:red, 0; green, 0; blue, 0 }  ,fill opacity=1 ] (519.6,150.98) .. controls (519.6,150.35) and (520.12,149.83) .. (520.75,149.83) .. controls (521.38,149.83) and (521.9,150.35) .. (521.9,150.98) .. controls (521.9,151.61) and (521.38,152.12) .. (520.75,152.12) .. controls (520.12,152.12) and (519.6,151.61) .. (519.6,150.98) -- cycle ;
\draw  [color={rgb, 255:red, 0; green, 0; blue, 0 }  ,draw opacity=1 ][fill={rgb, 255:red, 0; green, 0; blue, 0 }  ,fill opacity=1 ] (568.69,150.71) .. controls (568.69,150.08) and (569.2,149.56) .. (569.83,149.56) .. controls (570.47,149.56) and (570.98,150.08) .. (570.98,150.71) .. controls (570.98,151.34) and (570.47,151.85) .. (569.83,151.85) .. controls (569.2,151.85) and (568.69,151.34) .. (568.69,150.71) -- cycle ;
\draw  [color={rgb, 255:red, 0; green, 0; blue, 0 }  ,draw opacity=1 ][fill={rgb, 255:red, 0; green, 0; blue, 0 }  ,fill opacity=1 ] (589.85,161.34) .. controls (589.85,160.71) and (590.36,160.19) .. (590.99,160.19) .. controls (591.63,160.19) and (592.14,160.71) .. (592.14,161.34) .. controls (592.14,161.97) and (591.63,162.48) .. (590.99,162.48) .. controls (590.36,162.48) and (589.85,161.97) .. (589.85,161.34) -- cycle ;
\draw  [color={rgb, 255:red, 0; green, 0; blue, 0 }  ,draw opacity=1 ][fill={rgb, 255:red, 255; green, 255; blue, 255 }  ,fill opacity=1 ] (387.02,172.5) .. controls (387.02,171.04) and (388.2,169.85) .. (389.67,169.85) .. controls (391.13,169.85) and (392.32,171.04) .. (392.32,172.5) .. controls (392.32,173.96) and (391.13,175.15) .. (389.67,175.15) .. controls (388.2,175.15) and (387.02,173.96) .. (387.02,172.5) -- cycle ;
\draw  [color={rgb, 255:red, 255; green, 255; blue, 255 }  ,draw opacity=1 ][fill={rgb, 255:red, 255; green, 255; blue, 255 }  ,fill opacity=1 ] (581.85,164.85) .. controls (581.85,163.72) and (582.77,162.81) .. (583.9,162.81) .. controls (585.03,162.81) and (585.95,163.72) .. (585.95,164.85) .. controls (585.95,165.99) and (585.03,166.9) .. (583.9,166.9) .. controls (582.77,166.9) and (581.85,165.99) .. (581.85,164.85) -- cycle ;
\draw  [color={rgb, 255:red, 0; green, 0; blue, 0 }  ,draw opacity=1 ][fill={rgb, 255:red, 0; green, 0; blue, 0 }  ,fill opacity=1 ] (478.69,172.5) .. controls (478.69,171.87) and (479.2,171.35) .. (479.83,171.35) .. controls (480.47,171.35) and (480.98,171.87) .. (480.98,172.5) .. controls (480.98,173.13) and (480.47,173.65) .. (479.83,173.65) .. controls (479.2,173.65) and (478.69,173.13) .. (478.69,172.5) -- cycle ;
\draw  [color={rgb, 255:red, 0; green, 0; blue, 0 }  ,draw opacity=1 ][fill={rgb, 255:red, 0; green, 0; blue, 0 }  ,fill opacity=1 ] (568.69,172.5) .. controls (568.69,171.87) and (569.2,171.35) .. (569.83,171.35) .. controls (570.47,171.35) and (570.98,171.87) .. (570.98,172.5) .. controls (570.98,173.13) and (570.47,173.65) .. (569.83,173.65) .. controls (569.2,173.65) and (568.69,173.13) .. (568.69,172.5) -- cycle ;
\draw  [fill={rgb, 255:red, 0; green, 0; blue, 0 }  ,fill opacity=0.4 ] (569.83,150.71) -- (590.99,161.34) -- (569.67,172.5) -- cycle ;
\draw  [color={rgb, 255:red, 243; green, 161; blue, 25 }  ,draw opacity=0.8 ] (387.11,186.78) .. controls (387.11,191.45) and (389.44,193.78) .. (394.11,193.78) -- (440.13,193.78) .. controls (446.8,193.78) and (450.13,196.11) .. (450.13,200.78) .. controls (450.13,196.11) and (453.46,193.78) .. (460.13,193.78)(457.13,193.78) -- (518.78,193.78) .. controls (523.45,193.78) and (525.78,191.45) .. (525.78,186.78) ;
\draw  [color={rgb, 255:red, 243; green, 161; blue, 25 }  ,draw opacity=0.8 ] (595.33,139.08) .. controls (595.32,134.41) and (592.99,132.08) .. (588.32,132.09) -- (498.28,132.22) .. controls (491.61,132.23) and (488.28,129.9) .. (488.27,125.23) .. controls (488.28,129.9) and (484.95,132.23) .. (478.28,132.24)(481.28,132.24) -- (394.04,132.36) .. controls (389.37,132.37) and (387.04,134.7) .. (387.05,139.37) ;
\draw  [color={rgb, 255:red, 243; green, 161; blue, 25 }  ,draw opacity=0.8 ] (395.22,165.28) .. controls (395.25,163.56) and (394.41,162.68) .. (392.7,162.65) -- (392.7,162.65) .. controls (390.25,162.6) and (389.04,161.72) .. (389.07,160.01) .. controls (389.04,161.72) and (387.79,162.56) .. (385.34,162.51)(386.44,162.53) -- (385.34,162.51) .. controls (383.62,162.48) and (382.74,163.32) .. (382.71,165.03) ;
\draw  [color={rgb, 255:red, 243; green, 161; blue, 25 }  ,draw opacity=0.8 ] (188.83,230.31) .. controls (188.81,225.64) and (186.47,223.32) .. (181.8,223.34) -- (136.96,223.54) .. controls (130.29,223.57) and (126.95,221.26) .. (126.93,216.59) .. controls (126.95,221.26) and (123.63,223.6) .. (116.97,223.63)(119.97,223.62) -- (75.02,223.82) .. controls (70.35,223.85) and (68.03,226.19) .. (68.05,230.86) ;
\draw  [color={rgb, 255:red, 243; green, 161; blue, 25 }  ,draw opacity=0.8 ] (285.83,139.97) .. controls (285.82,135.31) and (283.49,132.98) .. (278.82,132.99) -- (183.86,133.24) .. controls (177.19,133.26) and (173.86,130.94) .. (173.85,126.27) .. controls (173.86,130.94) and (170.53,133.28) .. (163.86,133.29)(166.86,133.29) -- (75.03,133.53) .. controls (70.36,133.54) and (68.04,135.88) .. (68.05,140.55) ;
\draw  [color={rgb, 255:red, 255; green, 255; blue, 255 }  ,draw opacity=1 ][fill={rgb, 255:red, 255; green, 255; blue, 255 }  ,fill opacity=1 ] (581.85,234.85) .. controls (581.85,233.72) and (582.77,232.81) .. (583.9,232.81) .. controls (585.03,232.81) and (585.95,233.72) .. (585.95,234.85) .. controls (585.95,235.99) and (585.03,236.9) .. (583.9,236.9) .. controls (582.77,236.9) and (581.85,235.99) .. (581.85,234.85) -- cycle ;
\draw  [color={rgb, 255:red, 0; green, 0; blue, 0 }  ,draw opacity=1 ][fill={rgb, 255:red, 0; green, 0; blue, 0 }  ,fill opacity=1 ] (478.69,242.5) .. controls (478.69,241.87) and (479.2,241.35) .. (479.83,241.35) .. controls (480.47,241.35) and (480.98,241.87) .. (480.98,242.5) .. controls (480.98,243.13) and (480.47,243.65) .. (479.83,243.65) .. controls (479.2,243.65) and (478.69,243.13) .. (478.69,242.5) -- cycle ;
\draw  [color={rgb, 255:red, 0; green, 0; blue, 0 }  ,draw opacity=1 ][fill={rgb, 255:red, 0; green, 0; blue, 0 }  ,fill opacity=1 ] (568.69,242.5) .. controls (568.69,241.87) and (569.2,241.35) .. (569.83,241.35) .. controls (570.47,241.35) and (570.98,241.87) .. (570.98,242.5) .. controls (570.98,243.13) and (570.47,243.65) .. (569.83,243.65) .. controls (569.2,243.65) and (568.69,243.13) .. (568.69,242.5) -- cycle ;
\draw  [color={rgb, 255:red, 243; green, 161; blue, 25 }  ,draw opacity=0.8 ] (387.11,253.17) .. controls (387.11,257.84) and (389.44,260.17) .. (394.11,260.17) -- (484.01,260.17) .. controls (490.68,260.17) and (494.01,262.5) .. (494.01,267.17) .. controls (494.01,262.5) and (497.34,260.17) .. (504.01,260.17)(501.01,260.17) -- (588.11,260.17) .. controls (592.78,260.17) and (595.11,257.84) .. (595.11,253.17) ;
\draw  [color={rgb, 255:red, 0; green, 0; blue, 0 }  ,draw opacity=1 ][fill={rgb, 255:red, 0; green, 0; blue, 0 }  ,fill opacity=1 ] (410.69,242.5) .. controls (410.69,241.87) and (411.2,241.35) .. (411.83,241.35) .. controls (412.47,241.35) and (412.98,241.87) .. (412.98,242.5) .. controls (412.98,243.13) and (412.47,243.65) .. (411.83,243.65) .. controls (411.2,243.65) and (410.69,243.13) .. (410.69,242.5) -- cycle ;
\draw  [color={rgb, 255:red, 0; green, 0; blue, 0 }  ,draw opacity=1 ][fill={rgb, 255:red, 0; green, 0; blue, 0 }  ,fill opacity=1 ] (432.69,242.5) .. controls (432.69,241.87) and (433.2,241.35) .. (433.83,241.35) .. controls (434.47,241.35) and (434.98,241.87) .. (434.98,242.5) .. controls (434.98,243.13) and (434.47,243.65) .. (433.83,243.65) .. controls (433.2,243.65) and (432.69,243.13) .. (432.69,242.5) -- cycle ;
\draw  [color={rgb, 255:red, 0; green, 0; blue, 0 }  ,draw opacity=1 ][fill={rgb, 255:red, 0; green, 0; blue, 0 }  ,fill opacity=1 ] (590.69,242.5) .. controls (590.69,241.87) and (591.2,241.35) .. (591.83,241.35) .. controls (592.47,241.35) and (592.98,241.87) .. (592.98,242.5) .. controls (592.98,243.13) and (592.47,243.65) .. (591.83,243.65) .. controls (591.2,243.65) and (590.69,243.13) .. (590.69,242.5) -- cycle ;
\draw    (389.83,242.5) -- (433.83,242.5) ;
\draw    (479.83,242.5) -- (501.83,242.5) ;
\draw  [color={rgb, 255:red, 0; green, 0; blue, 0 }  ,draw opacity=1 ][fill={rgb, 255:red, 0; green, 0; blue, 0 }  ,fill opacity=1 ] (500.69,242.5) .. controls (500.69,241.87) and (501.2,241.35) .. (501.83,241.35) .. controls (502.47,241.35) and (502.98,241.87) .. (502.98,242.5) .. controls (502.98,243.13) and (502.47,243.65) .. (501.83,243.65) .. controls (501.2,243.65) and (500.69,243.13) .. (500.69,242.5) -- cycle ;
\draw  [color={rgb, 255:red, 0; green, 0; blue, 0 }  ,draw opacity=1 ][fill={rgb, 255:red, 0; green, 0; blue, 0 }  ,fill opacity=1 ] (568.52,242.5) .. controls (568.52,241.87) and (569.03,241.35) .. (569.67,241.35) .. controls (570.3,241.35) and (570.81,241.87) .. (570.81,242.5) .. controls (570.81,243.13) and (570.3,243.65) .. (569.67,243.65) .. controls (569.03,243.65) and (568.52,243.13) .. (568.52,242.5) -- cycle ;
\draw  [color={rgb, 255:red, 0; green, 0; blue, 0 }  ,draw opacity=1 ][fill={rgb, 255:red, 0; green, 0; blue, 0 }  ,fill opacity=1 ] (478.69,242.5) .. controls (478.69,241.87) and (479.2,241.35) .. (479.83,241.35) .. controls (480.47,241.35) and (480.98,241.87) .. (480.98,242.5) .. controls (480.98,243.13) and (480.47,243.65) .. (479.83,243.65) .. controls (479.2,243.65) and (478.69,243.13) .. (478.69,242.5) -- cycle ;
\draw  [color={rgb, 255:red, 0; green, 0; blue, 0 }  ,draw opacity=1 ][fill={rgb, 255:red, 0; green, 0; blue, 0 }  ,fill opacity=1 ] (388.69,242.5) .. controls (388.69,241.87) and (389.2,241.35) .. (389.83,241.35) .. controls (390.47,241.35) and (390.98,241.87) .. (390.98,242.5) .. controls (390.98,243.13) and (390.47,243.65) .. (389.83,243.65) .. controls (389.2,243.65) and (388.69,243.13) .. (388.69,242.5) -- cycle ;
\draw    (569.83,242.5) -- (591.83,242.5) ;

\draw (115.83,86.05) node [anchor=north] [inner sep=0.75pt]  [font=\fontsize{0.35em}{0.42em}\selectfont]  {$( 2,0,0)$};
\draw (200.75,55.43) node [anchor=south east] [inner sep=0.75pt]  [font=\fontsize{0.35em}{0.42em}\selectfont]  {$( 0,2,0)$};
\draw (276.14,70.34) node [anchor=west] [inner sep=0.75pt]  [font=\fontsize{0.35em}{0.42em}\selectfont]  {$( 0,1,0)$};
\draw (273.83,86.05) node [anchor=north] [inner sep=0.75pt]  [font=\fontsize{0.35em}{0.42em}\selectfont]  {$( 1,0,0)$};
\draw (253.67,48.1) node [anchor=south west] [inner sep=0.75pt]  [font=\fontsize{0.35em}{0.42em}\selectfont]  {$( 0,0,1)$};
\draw (90,14.07) node [anchor=north west][inner sep=0.75pt]  [font=\scriptsize]  {$\Delta _{1}$};
\draw (180,14.07) node [anchor=north west][inner sep=0.75pt]  [font=\scriptsize]  {$\Delta _{2}$};
\draw (270,14.07) node [anchor=north west][inner sep=0.75pt]  [font=\scriptsize]  {$\Delta _{3}$};
\draw (183.83,86.05) node [anchor=north] [inner sep=0.75pt]  [font=\fontsize{0.35em}{0.42em}\selectfont]  {$( 1,0,0)$};
\draw (20.67,167.32) node [anchor=north west][inner sep=0.75pt]  [font=\scriptsize] [align=left] {Bottom};
\draw (21.33,234.99) node [anchor=north west][inner sep=0.75pt]  [font=\scriptsize] [align=left] {Front};
\draw (340.67,168.32) node [anchor=north west][inner sep=0.75pt]  [font=\scriptsize] [align=left] {Side};
\draw (311.67,238.32) node [anchor=north west][inner sep=0.75pt]  [font=\scriptsize] [align=left] {Front-bottom};
\draw (492.99,272.2) node [anchor=north] [inner sep=0.75pt]  [font=\tiny]  {$1-3=-2$};
\draw (446.49,206.53) node [anchor=north] [inner sep=0.75pt]  [font=\tiny]  {$1-2=-1$};
\draw (390.99,151.53) node [anchor=north] [inner sep=0.75pt]  [font=\fontsize{0.35em}{0.42em}\selectfont]  {$0-1=-1$};
\draw (484.49,111.53) node [anchor=north] [inner sep=0.75pt]  [font=\tiny]  {$2-3=-1$};
\draw (168.66,115.53) node [anchor=north] [inner sep=0.75pt]  [font=\tiny]  {$2-3=-1$};
\draw (122.16,205.53) node [anchor=north] [inner sep=0.75pt]  [font=\tiny]  {$1-2=-1$};

\end{tikzpicture}

\caption{An example of a triple $\bdel$. For each of the four selected faces of $\bdel$, we indicate which are the important facings, and we compute their dimensions according to the formula~\eqref{eq:dimension}.}\label{fig:main}
\end{figure}
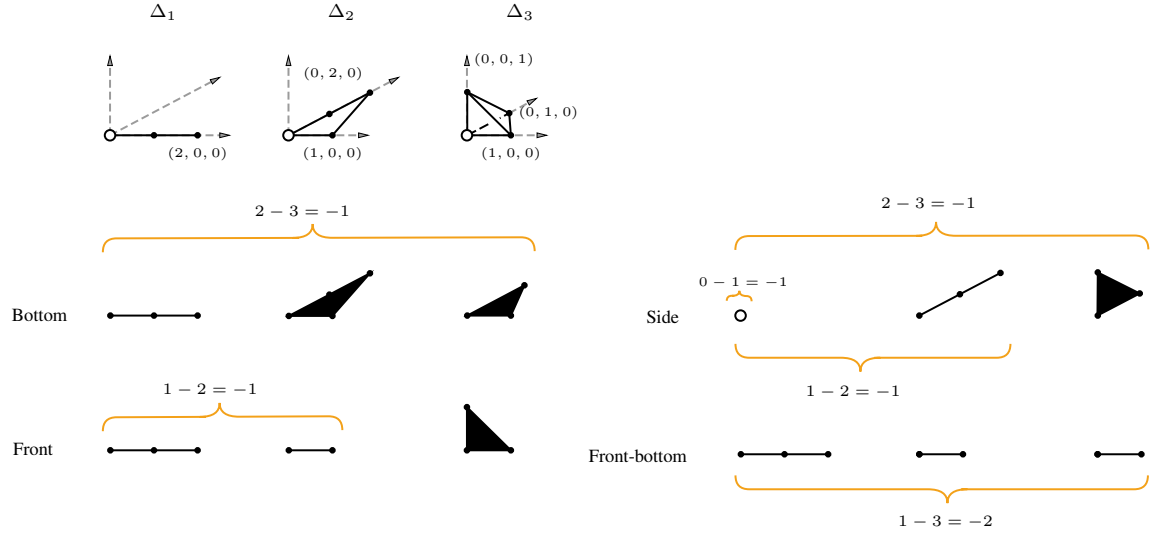

For any polytope $\Pi\subset\Rnn$, either $\{\unze\}$ is a vertex of $\Pi$, or $\unze\not\in \Pi$.

\begin{definition}\label{def:dicritical-facings}
A tuple $\bdel$ is said to be \emph{origin} if its support $[\bdel]$ contains the index $0$ and the polytope $\Delta_0$ contains the origin $\{\unze\}$.
\end{definition}
For example, all important facings illustrated in Figure~\ref{fig:main} are origin.

\subsection{Application to polynomial functions}
\label{sub:appl}
Let $\K \in \set{\C, \R}$. We consider the polynomial function
 $f:X\longrightarrow\K$, for some $X:=\VK(g_1,\ldots,g_r) = \VK(\g)$. If $f$ is obtained as the restriction $\Fr_X = f$, for some polynomial $F:\K^n\longrightarrow\K$, then, we identify $f$ with the collection of polynomials $\f:=(F,g_1,\ldots,g_r) = (F, \g)$. 

Let $\bdel$ denote the tuple of lattice polytopes given by $(\Delta_0,\Delta_1,\ldots,\Delta_r)$, where 
\[
\Delta_i:=\NP(g_i),~i=1,\ldots,r,
\] and $\Delta_0:=\NP(F-t)$ for some generic $t\in\K$. That is, it holds that
\[
\Delta_0:=\conv(\supp(F) \cup\{\unze\}).
\]
\begin{notation}\label{not:restricted}
For any subset $\sigma\subset\R^n$ and for any polynomial $P\in\K[x_1,\ldots,x_n]$ the \emph{restriction of $P$ to $\sigma$}, denoted by $P_{\sigma}$, is the polynomial 
\[
\sum_{\bma\in\sigma\cap\supp(P)}c_{\bma}~\x^{\bma}.
\]
\end{notation}

For any facing $\bgam\prec\bdel$, and any $\bmf\in\K^{\bdel}$, 
we define the collection of polynomials $\bmg_{\bgam}:=(\bmg_{\Gamma_i})_{i\in[\bgam]\setminus\{0\}}$,
where for any $i \in [\bgam]\setminus\{0\}$, the expression $\bmg_{\Gamma_i}$ denotes the restricted polynomial $g_{i,\Gamma_i}$. Similarly, we use the notation $\bmf_{\bgam}$ to refer to the tuple $(\bmf_{\Gamma_i})_{i\in[\bgam]}$ which includes $F_{\Gamma_0}$ whenever $0\in [\Gamma]$.

Let $\bgam$ be an origin facing for which $\VVKT(\bmg_{\bgam})$ is non-empty. Then it defines a function
\[
f_{\Gamma}:=\left.F_{\Gamma_0}\right|_{\VsK(\bmg_{\bgam})}:\VsK(\bmg_{\bgam})\longrightarrow \K,
\] where for any subset $S\subset\K^n$, we use the notation $S^*$ 
to denote the intersection $S\cap\TTnK$. 
We use this notation to define the set $\Rgf\subset\K$, called the \emph{face-discriminant} of $f$, defined as
%
\begin{equation}
\label{eq:face-discriminant}
\Rgf: =  f_\Gamma(\Crit(f_\Gamma))
\end{equation} 

In other words, the locus $\Rgf$ is the discriminant of the function $f_{\Gamma}$. By Bertini Theorem~\cite{Jean-Pierre_Jouanolou}, face-resultants are finite sets in $\K$. 
%

\begin{example}
Consider the polynomial function $F:\R^3\longrightarrow\R$, $(x,y,z)\longmapsto 1+ x + x^2$, and let $X\subset\R^3$ be the curve $\VR(g_1,~g_2)$, where $g_1:=  -2 + x +2y - y^2  $ and $g_2:=  1+2x - 3y + 4 z$. The triple $\bmf:=(F,g_1,g_2)$ has Newton tuple $\bdel$ illustrated in Figure~\ref{fig:main}. Let us compute some of the face-discriminants of $f$. 

If $\bgam=\bdel$, then we have $\Rgf = F(\Crit \Fr_{X^*})=\VR(h)$, where $h=\langle F-t,~g_1,~g_2,~\det\Jac_{(x,y,z)}\bmf\rangle$. 
 Another trivial case is whenever $\bgam$ is any one of the important facings of the ``Side'' face of $\bdel$ (see Figure~\ref{fig:main}); since $\Gamma_0 = \{(0,0,0)\}$, we always get $F_{\Gamma_0} \equiv 1$, and thus $\Rgf = \{1\}$.

Now let $\bgam\prec\bdel$ be the important origin face that is the ``Bottom'' triple of Figure~\ref{fig:main}. Hence, we have $F_{\Gamma_0}=F_{\Delta_0}=F$. Then, the domain of $\bmf_{\bgam}$ is the union of two veritcal lines $\VR(g_1,~1+2x - 3y)$. A generic $t\in \R$, has no preimages under $\bmf_{\bgam}$. This shows that 
\[
\Rgf = F(\VR(g,1+2x - 3y)) =\{1,~ 205/16\}. 
\]

\end{example}

\begin{remark}\label{rem:functional-case}
Note that, if $X=\K^n$, then we set $[\bdel] = \{0\}$. Hence, every facing $\bgam\prec\bdel$ corresponds to the first polytope $\Gamma_0$ that is a face of $\Delta_0$. By definition, we get 
\[
\Rgf = F_{\Gamma_0}(\Crit F_{\Gamma_0}).
\]
\end{remark}

\begin{notation}\label{not:impor-origin}
Let $\Ibdel$ and $\Obdel$ denote the set of all important facings of $\bdel$ and origin ones respectively. The intersection of the above two sets is denoted by $\IObdel$. 
\end{notation}

  We have the following classical result.

\begin{theorem}[\cite{NZ90}]\label{thm:Zaharia-Nemethi}
Let $\bdel$ be a tuple consisting of a single integer polytope in $\R^n_{\geq 0}$. Then, 
there exists a Zariski open subset $\Upsilon\subset\C^{\bdel}$, such that for every polynomial $f\in \Upsilon$ it holds that
\[
\cBf\setminus \{f(\unze)\}\subset\bigcup_{\bgam\in\IObdel}\Rgf.
\] 
\end{theorem}
The detailed version of Theorem~\ref{thm:Zaharia-Nemethi} illustrates a description of the set $\Upsilon$, together with all the origin faces of $\Delta$ that do not contribute to the infinity bifurcation set. This demonstrates that face-resultants are useful for effectively computing the bifurcation set of a large family of complex polynomials. 
 
Let us state a generalization of Theorem~\ref{thm:Zaharia-Nemethi}, and postpone its proof to the end of this section.

\begin{theorem}\label{thm:face-generic_tuple}
Let $\bdel$ be a tuple of lattice polytopes in $\Rnn$ satisfying $[\bdel] = \{0,1,\ldots,r\}$, and assume that $\{\unze\}$ is a vertex of $\Delta_0$. Then, there exists a Zariski open subset $\Omega\subset\K^{\bdel}$, such that for every $\bmf\in \Omega$, the set $\VK(g_1,\ldots,g_r)$ is smooth, and the correcponding polynomial function $f:\VK(g_1,\ldots,g_r)\longrightarrow\K$ satisfies
\begin{equation}\label{eq:bifurcation_generic}
\cBf\setminus \{f(\unze)\}\subset \bigcup_{\bgam\in\IObdel} \Rgf .
\end{equation}
%
\end{theorem}

In what follows, we use $\fstar$ to denote the function
\begin{align*}
\fstar:= \left. f\right|_{X^*}: &~X^*\longrightarrow\K.
\end{align*}

Similarly to $\cBf$, the set $B_{\fstar}$ is also finite~\cite{Tho69}. 

%
%

\begin{lemma}\label{lem:affine_then_very-affine}
We have
\begin{align}\label{eq:bifurcation_set_hyperplanes}
\cBf & = \bigcup_{H}\cB_{\left. f\right|_{H}}~\cup \cB^{\infty}_{\fstar}\cup \cBfz,
\end{align} where $H$ runs over all coordinate hyperplanes in $\K^n\setminus\TTnK$, and $\fr_{H}$ are the restricted functions $X\cap H\longrightarrow \K$.
%
%
\end{lemma}
\begin{proof}
We consider only the case where $X\neq \TTnK$ as the complementary case is similar. Let $\lambda\in\cBf$, and assume that $\lambda\not\in\cB^{\infty}_{\fstar}\cup \cBfz$. Without loss of generality we set $\lambda=0$. Then, for every arbitrary large compact subset $K\subset \K^n$, there is a small enough neighborhood $U \ni 0$, for which the function $h:\fstar^{-1}(U)\setminus K\longrightarrow U$, obtained by restricting $\fstar$ to $\fstar^{-1}(U)\setminus K$, is a $\cC^{\infty}$-fibration. Consequently, for each $z\in U$, we have $h^{-1}(0)$ and $h^{-1}(z)$ are diffeomorphic. Then, from $0\not\in \cBfz:=f(\Crit f)$, the two preimages are diffeomorphic:
\begin{align}\label{eq:closures-diffeo}
\overline{h^{-1}}^{_X}(0)\cong &~\overline{h^{-1}}^{_X}(z),
\end{align} 
where  the closures are taken in the Eucludean topology induced on $X$. However, there exists a value $z\in U$ for which the following holds
\begin{align}\label{eq:closures-diffeo2}
f^{-1}(0)\setminus K\not\cong &~f^{-1}(z)\setminus K.
\end{align} We deduce from~\eqref{eq:closures-diffeo} and~\eqref{eq:closures-diffeo2} that for some coordinate hyperplane $H$, the preimage $(f^{-1}(0)\setminus K)\cap H $ is not diffeomorphic to $(f^{-1}(z)\setminus K)\cap H $ for $K$ large enough. This shows that $z\in\cB_h$.
\end{proof}

In order to prove Theorem~\ref{thm:face-generic_tuple}, we furthermore require the following notion for tuples $\bmf\in \K^{\bdel}$. Taking the coefficients of $\bmf$ in the parameter space $\K^{\bdel}$ as additional variables, we denote by $\{\bmf=0\}$ the locus $\{F = g_1=\cdots= g_r = 0\}\subset\TTnK\times\K^{\bdel}$. The \emph{Bertini discriminant} $\bm{B}$ of $\bdel$ (see e.g.,~\cite[Definition 3.5]{Est13}) is the bifurcation set of the map $\left.\pi\right|_{\{\bmf=\unze\}}$ 
by restricting to $\{F = g_1=\cdots= g_r = 0\}$ the projection 
\[
\pi:\TTnK\times\K^{\bdel} \longrightarrow \K^{\bdel}.
\] 

The Bertini discriminant can be computed as follows. For any facing $\bgam\prec\bdel$, we define $\Dbg\subset \K^{\bdel}$ as the closure of 
\begin{align}\label{eq:dgamma}
\left\lbrace \bmf~\left|~\exists \x_0\in \VsK(\bmf_{\bgam}),~\left.\Jac_{\x}(\bmf_{\bgam})\right|_{\x = \x_0} \text{ does not have full rank}\right\rbrace\right.
\end{align} In other words, the set $\Dbg$ consists of all $\bmf\in \K^{\bdel}$ for which there exists a point $\x_0\in\TTnK$ at which $\VsK(\bmf_{\bgam})$ is not a complete intersection. It was shown in~\cite[Theorem 1.1 and Proposition 4.10]{Est13} (see also~\cite[Section 5]{GKZ94}) that $\cdim\bm{B}=1$, and it holds that
\begin{align}\label{eq:Bertini_unions}
\bm{B} = &~ \bigcup_{\bgam\in\Ibdel} \Dbg.
\end{align}

\begin{remark}\label{rem:Cayley-real}
The mentioned results in~\cite{Est13} are formulated for the case where $\K=\C$. For the real case, the statement assumes that $\bmf$ consists of only one polynomial $F$~\cite[Section 5.A]{GKZ94}. Transitioning to arbitrary tuples $\bmf$ with $\K=\R$ can be done via the Cayley Trick (see e.g.,~\cite[Section 6.2]{Est10}).
\end{remark}
\begin{proof}[Proof of Theorem~\ref{thm:face-generic_tuple}]
Let $\bmf\in\K^{\bdel}\setminus\bm{B}$, and let $\Lambda(\bmf)\subset\K^{\bdel}$ be the line passing through $\bm{f}$ with direction $(1,0,\ldots,0)$, where the first coordinate corresponds to the constant term of $F$. In other words, we have 
\begin{align*}
\Lambda(\bmf):= & \left\lbrace (F-z,g_1,\ldots,g_r)~|~z\in\K\right\rbrace. 
\end{align*} On the one hand, if $\Theta\subset\K^\Delta$ is the set of all $\bmf$ at which $\Lambda(\bmf)$ intersects $\bm{B}$ transversally, then for each $\bmf\in\Theta$, the relation
\begin{align}\label{eq:Bertini-bifurcation}
\cB_{\fstar}= &~\bm{B}\cap \Lambda(\bmf),
\end{align} follows from the definition of $\bm{B}$. Note that if $\widetilde{\bdel}$ is the tuple of polytopes $(\widetilde{\Delta}_0,\Delta_1,\ldots,\Delta_r)$, where 
\begin{align*}
\widetilde{\Delta}_0:= & \conv\left( \Delta_0\cap\N^n\setminus\{\unze\} \right),
\end{align*} then $\Theta$ is the preimage, under $\Pi$, of the bifurcation set of the map $\left. \Pi\right|_{\bm{B}}$, where $\Pi:= \K^{\bdel}\longrightarrow \K^{\widetilde{\bdel}}$ is the projection taking a tuple $\bmf$ to $\tilde{\bmf}$ by forgetting the constant term of the first polynomial $F$. Thanks to a theorem of Verdier~\cite{verdier1976stratifications}, the set $\Theta$ contains a Zariski open subset $\Omega\subset\K^\Delta$.

On the other hand, for any $\bmf\in\K^\Delta$, and each facing $\bgam\prec\bdel$, we have 
\begin{align}\label{eq:R=DL}
\Rgf = &~ \Dbg\cap\Lambda(\bmf).
\end{align} Therefore, since $\Ibdel$ contains the tuple $\bdel$, it follows from~\eqref{eq:Bertini_unions},~\eqref{eq:Bertini-bifurcation}, and~\eqref{eq:R=DL} that
\begin{align}\label{eq:bbdI}
\cB_{\fstar}= &~\bigcup_{\bgam\in\Ibdel} \Rgf.
\end{align} Next, we show that we can replace the set $\Ibdel$ in~\eqref{eq:bbdI} by $\IObdel$. Recall that $\bmg_{\bgam}$ refers to $(\bmg_{\Gamma_i})_{i\in[\bgam]\setminus\{0\}}$, and $\bmf_{\bgam}$ to $(\bmf_{\Gamma_i})_{i\in[\bgam]}$ which includes $F_{\Gamma_0}$ whenever $0\in [\Gamma]$.

The set $\Dbg$ is a fiber bundle $\Proj^{-1}(\Proj(\Dbg))$, where $ \Proj:\K^{\bdel}\longrightarrow \K^{\bgam}$, $\bmf\longmapsto\bmf_{\bgam}$. Hence, since $\bmf\not\in \Dbg$, we have $\bmf_{\bgam}\not\in\Proj(\Dbg)$. This shows that $\Proj^{-1}(\bmf_{\bgam})\cap\Dbg = \emptyset$. Then, if $\bgam$ is not origin, we get $\Lambda(\bmf)\subset\Proj^{-1}(\bmf_{\bgam})$. We conclude that $\Rgf = \emptyset$ if $\bgam$ is not origin, and thus we get 
\begin{align}\label{eq:bbdIO}
\cB_{\fstar}= &~\bigcup_{\bgam\in\IObdel} \Rgf.
\end{align} 

It remains to prove that for any $z\in \cBf\setminus (\cB^{\infty}_{\fstar}\cup\cBfz)$, we have $z\in\Rgf$ for some $\bgam\in\IObdel$. By Lemma~\ref{lem:affine_then_very-affine}, we have $z\in  \cB_h^{\infty}$ for the restricted function $h:=\left. f\right|_{\{x_i = 0\}}:X\cap \{x_i = 0\}\longrightarrow \K$ for some $i\in [n]$. Without loss of generality, we assume that $i=n$. Notice that $g$ can also be expressed as the function
\begin{align}\label{eq:f-restr=g}
f_{\Gamma'}:= \left. F_{\Gamma_0'}\right|_{_{X\cap\{x_n=0\}}}\V_\K(\bmg_{\bgam'})\longrightarrow \K,
\end{align}
where $\bgam'\prec\bdel$ is the facing whose Minkowski sum spans the coordinate hyperplane $\{x_n = 0\}\subset\R^n$. Indeed, we deduce this by plugging $x_n = 0$ into the equations of $\bmf$. 

By computing the Jacobian matrix, we deduce that $z\not\in\cB_f^0\Longrightarrow z\not\in\cB_h^0$. Then, thanks to Lemma~\ref{lem:affine_then_very-affine} and~\eqref{eq:bbdIO} applied to $h$, if $z\not\in\cR_{\Gamma}(h) $ for some important origin facing $\bgam\prec\bgam'$, then $z\in \cB_\ell$, where $\ell:= \left. h\right|_{_{\{x_j=0\}}}:X\cap \{x_jx_n = 0\}\longrightarrow\K$ for some $j\in [n-1]$. Notice that $\bgam$ is an important facing of $\bdel$, and that $\Rgf  = \mathcal{D}_{\Gamma}(h)$.

Applying recursively the above arguments, we deduce that either $z\in \Rgf $ for some $\bgam\in\IObdel$, or there exists $\widetilde{\bgam}\in\IObdel$ such that $z\in\cB_{f_{\widetilde{\Gamma}}}$, where $\widetilde{\Gamma}_0 = \{\unze\}$, and 
\begin{align}\label{eq:f-restr=g2}
f_{\widetilde{\Gamma}}:= \left. F_{\widetilde{\Gamma}_0}\right|_{_{X\cap\{x_1\cdots x_n=0\}}}\V_\K(\bmg_{\widetilde{\bgam}})\longrightarrow \K.
\end{align}
Clearly, we have $f_{\widetilde{\Gamma}}$ is the trivial map whose target space is the point $\{F(\unze)\}$. This finishes the proof.
\end{proof}

\begin{remark}\label{rem:1}
From the proof of Theorem~\ref{thm:face-generic_tuple}, in combination with Eq.~\eqref{eq:bbdIO}, we can deduce that $\cBf\subset\cB_{\fstar}$ if $\bmf\in\Omega$.
\end{remark}

\subsection{Bounding the infimum}\label{sub:bounding_infimum}

To obtain all the asymptotic critical values it suffices to go through all the face lattice of  $\bdel$ and compute the corresponding 
face discriminant of $f$, see Eq.~\eqref{eq:face-discriminant}.
As we aim for worst case bounds, we consider the worst case where all the polynomials are of degree $d$. That is we ignore their restriction to facings.
We emphasize, once more, that is suffices for our purposes as we are targeting an algorithm but bounds on the bitsize on the corresponding quantities. 

The face discriminant results in a polynomial system
of at most $\OO(s + n^n)$ 
polynomials; that is the $s$ polynomials $g_i$, the $\OO(n^n)$ polynomials coming from the rank condition on the Jacobian, and the polynomial $F - z$.
The bitsize of the polynomials is at most $\sOO(n \tau)$; dominated by the polynomials from the rank condition on the Jacobian.
Our goal is to bound the roots of this system, which is 0-dimensional, due to Bertini's theorem. 
In particular, we want to bound $z$, that corresponds to the asymptotic critical values and consequently to $f^{*}$.
Following \cite{emt-dmm-j}, the bounds appearing in Theorem~\ref{thm:main_Newton_non-degenerate} hold true for any coordinate of the roots of the system and for $f^{*}$.

\section{Unconstrained optimization and asymptotic critical values}
\label{sec:uncon-opt}

We present an algorithm (Alg.~\ref{alg:ACV}) to compute (and to obtain precise bounds on) the elements of the Rabier set $\cK(f)$, that is the set of generalized critical values, of a polynomial function $f : \CC^n \to \CC$, where $f \in \ZZ[\x]$. 
	One of these values corresponds to the global infimum of the (polynomial) optimization problem $f^* = \inf_{\x \in \RR^n} f(\x)$. 
We assume that $f$ has degree $d$ and maximum coefficient 
bitsize $\tau$, or in short that $f$ is of size $(d, \tau)$.

\subsection{Asymptotic critical values over {$\CC^n$}}
\label{sec:acv-Cn}

The (pseudo-code of the) algorithm in Alg.~\ref{alg:ACV} supports the computations of the asymptotic critical values of $f$ 
and is based on \cite{JelTib17}.
The main difference is the way that we perform the elimination. Instead of Gr\"obner
basis we use resultant matrices to control better the bitsize of the various objects
and the complexity of the overall algorithm.
The algorithm relies on the following notion.

\begin{definition} 
	\label{def:super-polar-curve}
For any $f\in\C[x_1,\ldots,x_n]$, we define a set $\{g_1,\ldots,g_{n-1}\}\subset\C[x_1,\ldots,x_n]$ of polynomials given by 
\[
g_k:= \sum_{i=1}^{n} a^{(k)}_i \frac{\partial f}{\partial x_i} 
        + \sum_{i,j =1}^{n} b^{(k)}_{i,j} x_i \frac{\partial f}{\partial x_j},\quad\text{for } k \in [n-1],
\] 
for some choice of coefficients $(\bma,\bmb):=(\bma^{(1)},\bmb^{(1)},\ldots,\bma^{(n-1)},\bmb^{(n-1)})\in\C^{n^3-n}$. Then, the \emph{super polar curve} of $f$ is the set 
\begin{align*}
\cG_f(\bma,\bmb) := & \overline{\V(g_1,\ldots,g_{n-1}) \setminus\Crit(f)},
\end{align*} where we use here the Zariski closure. The super polar curve is said to be \emph{non-degenerate} if $\cG_f(\bma,\bmb)\subset\C^{n}$ is indeed a curve, i.e., of dimension 1.
\end{definition}
Roughly speaking, super polar curves encapsulate the behaviour at infinity of the image of (sequence of) points in $\CC^n$ under $f$ when their norm tends to infinity. A key point of the algorithm is to ensure that the linear combinations $g_k$ are such that the resulting set $\cG_f(\bma,\bmb)$ is non-degenerate.

\begin{algorithm2e}[ht]
  \SetFuncSty{textsc} \SetKw{RET}{{\sc return}} \SetKw{OUT}{{\sc output \ }} \SetKwInOut{Input}{Input}
  \SetKwInOut{Output}{Output}
    \SetKwInOut{Require}{Require}
    \Input{$f \in \ZZ[x_1,\ldots, x _n]$}
    \Output{A set $K \in \CC$ such that $\Kinf(f) \subset K$.}

    \BlankLine

    Choose generic $a^{(k)}_i$ and $b^{(k)}_{i,j}$, for $k \in [n-1]$ and $i, j \in [n]$
    
    \For{$k \in [n-1]$} {
        $g_k(\x)= \sum_{i=1}^{n} a^{(k)}_i \frac{\partial }{\partial x_i}f(\x) 
        + \sum_{i,j =1}^{n} b^{(k)}_{i,j} x_i \frac{\partial}{\partial x_j}f(\x)$ \;
    }
    
    \For{$i \in [n]$}{
	  $\bar{h}_i \gets \mathtt{Elim}(\{g_1(\x), \dots, g_{n-1}(\x), f(\x)-z \}, \x_{-i}) \in \ZZ[x_i, z]$ \;
	  $h_i \gets \lcoeff(\bar{h}_i, x_i)$ \;
    }
    
    $h \gets \prod_{i=1}^{n} h_{i}(z) \in \ZZ[z]$ \;


    \BlankLine
    \RET  $K := \{\gamma \in \CC \,|\, h(\gamma) = 0 \}$  \tcc{The distinct roots of $h \in \ZZ[z]$}
    \caption{\textsc{AsymptoticCriticalValues$(f)$}}
  \label{alg:ACV}
\end{algorithm2e}

To study the complexity of Alg.~\ref{alg:ACV},
the bitsize of the
elements of a linear transformation that we can apply to $f$ and its derivatives to obtain the polynomials $g_k$, such that the corresponding super-polar curve is non-degenerate.
This bound permits us to estimate the probability of success
when we perform random linear combinations.

Nevertheless, if needed, we can always obtain a Las Vegas algorithm,
that is an algorithm that always returns a correct output, if we allow ourselves the cost of certifying the result; that is to check if the resulting super-polar curve is indeed non-degenerate.
%
It was shown in~\cite{JelTib17} that Alg.~\ref{alg:ACV} computes correctly the Rabier set of a polynomial $f$. Proving our main result, Theorem~\ref{thm:unconstrained_detailed}, relies on the complexity analysis of Alg.~\ref{alg:ACV}.
For this we need two technical results. The first is the following proposition, the proof of which is in the Appendix (Lemma~\ref{lem:projV-intersect-H}).

\begin{proposition}
  \label{prp:affV-intersect-H}
  Consider an affine variety $V \in \CC^n$
  of dimension $\delta$ and degree $D$.
  Let $S$ be finite set of numbers such that $0 \not\in S$.
  Also consider the linear form $h(\x) = c_0 + \sum_{i=1}^nc_i x_i$ and $H = \VV(h)$.
  If we choose the $c_i$'s independently and uniformly at random from $S$, then
  \[ \Pr[ \dim(V \cap H) = \delta - 1 ] \geq 1 - \tfrac{D}{\abs{S}}.
  \]
  Moreover, there is a specialization of the $c_i$'s to integers  such that
  their bitsize is $\lceil \lg( 2 \kappa D) \rceil + 1$, where $\kappa \geq 2$ is a constant,
  so that it holds $\dim(V \cap H) = \delta -1$.
\end{proposition}

The the second technical result is as follows.

\begin{lemma}[Genericity of the super polar curve]
  \label{lem:lin-transform-super-polar}

 For any $f\in\Z[x_1,\ldots,x_n]$, there are integers $(\bm{a}, \bm{b}) \in \ZZ^{n^{3}-n}$ of bitsize $\OO(n^{2} + \lg{d})$ such that the super-polar curve $\Gfab$ is non-degenerate.

  In
  addition, if we pick the elements of $(\bm{a}, \bm{b})$ uniformly at random
  from a set $S$ containing $2~cnd^{n-1}$ elements, where $c$ is a constant,
  then the super-polar curve is non-degenerate with probability
  $\geq 1 - \tfrac{1}{c}$.
\end{lemma}
\begin{proof}

Assume that each of the polynomials The polynomials $g_{1},\ldots,g_{n-1} \in (\ZZ[\bm{a}, \bm{b}])[\bm{x}]$ are of degree at most $d$ with respect to $\x$.

Following \cite{JelTib17}, we know that if $(\bm{a}, \bm{b})$
lies in a Zariski open set $\cZ\subset\CC^{n^3-n}$, then $\Gfab$ is non-degenerate.


Then, thanks to Proposition~\ref{prp:affV-intersect-H}, there are linear
polynomials, say $L_0, L_n \in \ZZ[\x]$,
where $L_{i} = \ell_{i,0} + \sum_{j=1}^{n}\ell_{i,j}x_{j}$ for $i \in \{0, n\}$,
such that zero set of the
following polynomial system 
\begin{equation}
  \label{eq:sys-Gamma}
  \{ L_{0} = g_{1} =  \dots =  g_{n-1} =  L_{n} = 0\},
\end{equation}
has no solutions for any $(\bma,\bmb)\in\cZ$.

The coefficients of the linear polynomials $L_0$ and $L_n$
are integers, but for the moments we consider them as (new) parameters.
The system~\eqref{eq:sys-Gamma} consists of $n+1$ polynomials in $n$ variables,
i.e., $\{x_{1}, \dots, x_{n}\}$.

Let $\bm{\ell}$ denote the set of cofficients of $L_{0}$ and $L_n$.
Eliminating the variables $\x$ from the system~\eqref{eq:sys-Gamma} we obtain an ideal generated by a polynomial $R\in\C[\bm{\ell},\bma,\bmb]$, called the \emph{resultant}. Note that $R$ is not identically zero. Indeed, from $(\bma,\bmb)\in\cZ$, we get $\dim(\Gfab) = 1$ and thus Proposition~\ref{prp:affV-intersect-H}
guarantees that there are (two) linear polynomials, hence a specialization of $\bm{\ell}$,
such that
the system in~\eqref{eq:sys-Gamma} has no solutions. Therefore,  for any choice $(\bm{\ell},\bma,\bmb)$ such that $ (\bma,\bmb)\in\cZ$ and $R(\bm{\ell},\bma,\bmb)\neq 0$, we have $\dim \Gfab = 1$ and~\eqref{eq:sys-Gamma} has no solutions.

Notice that $R \in (\ZZ[\bm{c}, \bm{\ell}])[\bm{a}, \bm{b}]$
and as a polynomials in $\bm{a}$, or $\bm{b}$
it has at most $n^{3} - n$ variables
and its total degree is at most  $n d^{n-1}$.




Following the DeMillo-Lipton-Schwartz-Zippel (DLSZ) lemma~\cite{lipton,saxena,sharir},
if $S$ is a finite set of integers, then $R$ has at most
$2 n d^{n-1} \abs{S}^{n^{3}-n-1}$ solutions in the grid $S^{n^3 - n}$.

If $S$ contains the first $2~c  n d^{n-1}$ integers (where $c$ is a constant
greater than 2), then we deduce that there is a specialization of $( \bm{a}, \bm{b})$ such
that $\dim(\Gamma) = 1$. Also the bitsize of the elements of $\bm{a}$ and $\bm{b}$
is  $\OO(n^2 + \lg(d))$.

Moreover, thanks to (DLSZ), if we pick a specialization for $(\bm{a}, \bm{b})$ uniformly at random
in the grid $S^{n^3 - n}$, then $\dim(\Gamma) = 1$
with probability $\geq 1 - \tfrac{1}{c}$.
\end{proof}

\begin{theorem}
  \label{thm:gen-asympt-computation}
  Let $f \in \ZZ[\x]$ be of (total) degree at most $d$ and bitsize $\tau$. Then,
  Alg.~\ref{alg:ACV} is a Monte Carlo algorithm that computes the
  asymptotic critical values in
  $\sOB( n^{\omega n - \omega + 1} d^{(\omega + 3)n - \omega -1} \tau)$
  bit operations,
  where $\omega$ is the exponent of the complexity of matrix multiplication.

	The asymptotic critical values are algebraic numbers of degree $\OO(d^{n-1})$ and the (sum of the) bitsize(s) of their isolating intervals (or boxes) is 
  $\sOO(n d^{2n-2}\tau + n^{2}d^{n-1})$.
 \end{theorem}
\begin{proof}
  First we construct the polynomials $g_{k}$, for $k \in [n-1]$;
  they are of degree $d$.
  Following Lemma~\ref{lem:lin-transform-super-polar}
  the bitsize of elements of $\bm{a}$ and $\bm{b}$ is $\OO(n^{2} + \lg{d})$,
  hence the bitsize of $g_{k}$ is $\OO(\tau + n^{2} + \lg{d})$.

  Then we consider the set of polynomials
  \begin{equation}
	\label{eq:sys-gk-fz}
	\{ g_1, \dots, g_{n-1}, f -z \},
  \end{equation}
  and we eliminate the variables
  $x_1, \dots, x_{i-1}, x_{i+1}, \dots, x_n$, for $i \in [n]$.
  Thus, we perform the elimination $n$ times.

	Without loss of generality, we may assume that $i = n$.


  After eliminating the variables $x_1, \dots, x_{n-1}$
  from the set of polynomials in \eqref{eq:sys-gk-fz},
  we obtain a polynomial $R_n \in \ZZ[x_n, z]$; this is the resultant of the
  polynomials in \eqref{eq:sys-gk-fz}.
  We compute $R_n$ using the Macaulay matrix, say $M$, corresponding
  to the system in \eqref{eq:sys-gk-fz}.
  In particular we express $R_n$ as the ratio of two determinant of
  (sub)matrices of $M$.
  It might happen that the determinant in the denominator is zero.
  To avoid this problem and still compute $R_n$ we use the technique of
  generalized characteristic polynomial \cite{canny_generalised_1990}.
  For this we perturb symbolically the polynomials $g_i$
  to become $G_i = g_i + s x_i^d$, where $s$ is a new variable and $i \in [n-1]$;
  and $G_n = f -z$.
  In this way, the ratio of determinants from $M$ result in a polynomial $S_n \in \ZZ[x_n, z][s]$.
  The non-vanishing coefficient of $S_n\in\Z[s]$, in front of the monomial with the smallest degree in $s$, 
  is a power of $R_n$.

  To obtain bounds on $S_n$ and $R_n$ we rely on \cite{emt-dmm-j}.
  Following \cite{CLO2}, $S_n$ is homogeneous
  in the coefficients of the each polynomial $G_i$.
  In particular, each term of $S_n$ is as follows
  \[
	\varrho_k \abs{\bm{c}_1}^{d^{n-1}} \cdots \abs{\bm{c}_n}^{d^{n-1}},
  \]
  where the semantics of $\abs{\bm{c}_i}^{d^{n-1}}$ are 
  that it represents a product of coefficients of the polynomial $G_i$ all of which have total degree $d^{n-1}$.
  The coefficients of $G_i$ are in $\ZZ[x_n, z, s]$, they have multidegree $(d, 0, 1)$, 
  for $i \in [n-1]$ and $G_n$ has multidegree $(d, 1, 0)$;
  their bitsize $\tau$.
  Hence,  $\abs{c_i}^{d^{n-1}}$ is of multidegree $(d^n, 0, d^{n-1})$ for $i \in [n-1]$
  and $G_n$ has multidegree $(d^n, d^{n-1}, 0)$;
  their bitsize $\OO(d^{n-1}\tau + d^{n-1}\lg(d))$.
  Also $\lg \abs{\varrho_k} = \OO(n d^{n-1} \lg(d))$.
  Consequently, each term is a polynomial in $\ZZ[x_n, z, s]$
  of multidegreee $(d^n, d^{n-1}, d^{n-1})$
  and bitsize $\OO(n d^{n-1}\tau + n d^{n-1}\lg(d) + n^2 \lg(n d))$,
  using Claim~\ref{claim:mul-mpoly}.
  As there are at most $\OO(d^{3n})$ terms, the bitsize of $S_n$ and $R_n$
  is $\OO(n d^{n-1}\tau + n d^{n-1}\lg(d) + n^2 \lg(n d))$.

  To actually compute $S_n$  and $R_n$ we use
  perform the elimination using the Macaulay matrix.
  We consider the homogenization of the polynomials $G_1, \dots, G_n$.
  be introducing a new variable $x_0$.

  Let $m = \sum_{i=1}^{m}(d-1) + 1 = n(d-1)+1 \leq nd$ and
  $N = {m +n -1 \choose n-1} \leq (nd)^{d-1}$.
  The latter corresponds to all monomials in $n$ variables of degree $\leq m$.

  As the elements of $M$ are polynomials in $\ZZ[x_n, z, s]$ we apply Kronecker's trick
  to obtain univariate polynomials.
  In particular, we perform the transformation
  $x_n \gets s^{d^{n-1}+1}$ and $z \gets s^{(d^{n-1}+1)d^n}$.
  Then, the elements of $M$ become polynomials in $\ZZ[s]$
  of degree $\leq d^{2n-1} + d^n +1 = \OO(d^{2n})$.
  Now, we can compute the determinants
  $\OO(N^{\omega} d^{2n}) = n^{\omega(n-1)} d^{(\omega + 2n) - \omega}$
  arithmetic operations.

  If we account on the bitsize of $S_n$ (and $R_n$) we deduce
  that we can compute them
  $\sOB( n^{\omega n - \omega + 1} d^{(\omega + 3)n - \omega -1} \tau)$
  bit operations.

  Hence, after we arrange the terms of the resultant, we can recover the polynomial
  and $\bar{h}_{n}$ and thus $h_{n} \in \ZZ[z]$.
  The latter has size $(d^{n-1}, \sOO(n d^{n-1}\tau + n^{2}))$.

  We isolate its real (or complex) real roots
  in $\sOB(n d^{3n-3}\tau + n^{2}d^{2n-2})$ \cite{Pan-usolve-02}.
  The (sum of the) bitsize of the isolating interval(s) (or boxes if we consider the complex roots)
  of the roots is $\sOO(n d^{2n-2}\tau + n^{2}d^{n-1})$, e.g., \cite{emt-dmm-j}.
\end{proof}

\subsection{Polynomial optimization over $\RR^n$}
\label{sec:opt-over-Rn}
In this part, we show how to use the Rabier set of $f$~\eqref{eq:bifur-Rabier}, together with Alg.~\ref{alg:ACV}, for computing the global infimum over $\RR^n$. 
Following Theorem~\ref{thm:bifur-semi} (see also, \cite{Safey-opt-R-2008}), it holds that
\begin{equation}
	\label{eq:minimum_in_critical}
	\inf_{x\in\RR^n}f(x)\subset \Kall(f)  \cap\RR,
\end{equation}
where $\Kall(f) = \Kall_0(f) \cup \Kinf(f)$ is the Rabier set of critical values of $f$. It remains to identify which 
of these values corresponds to the infimum of $f$.

If the Rabier values are $e_{1}, \dots, e_{m}$,
then we need to compute rational numbers that interlace them~\cite{Safey-opt-R-2008}, that is $r_{i}$
for $0 \leq i \leq m$, such that
\[
  r_{0} < e_{1} < r_{1} < \dots < r_{m-1} < e_{m} < r_{m}
\]
and test if the real hypersurfaces $H_i:= f^{-1}(r_{i})\cap\R^n$ are empty or not.
The bitsize of corresponding polynomial defining $H_i$ is dominated by the bitsize of $r_{i}$,
let it be $\sigma_{i}$. 
Thanks to Theorem~\ref{thm:gen-asympt-computation}, we have 
\[
\sum_{i=0}^{m}\sigma_{i} =  \sOO(n d^{2n-2}\tau + n^{2}d^{n-1}) .
\]

Next, we need the following theorem.
\begin{theorem}[\cite{elliott-bit-2020}]
  \label{thm:test-emptiness}
  For $f \in \ZZ[\x]$ of size $(d, \tau)$ there is a randomized algorithm that
  with probability $1 - \epsilon$ decides if $\VV(f) \cap \RR^{n}$ is empty or
  not using
  $\sOB(d^{3n}\lg(\tfrac{1}{\epsilon})(\tau + \lg(\tfrac{1}{\epsilon})))$ bit
  operations.
\end{theorem}

Using Theorem~\ref{thm:test-emptiness} we can test the emptiness of all the real hypersurfaces in
\[
  \sum_{i=0}^{m}\sOB(d^{3n}\lg(\tfrac{1}{\epsilon})(\sigma_{i} + \lg(\tfrac{1}{\epsilon})))
  = \sOB(d^{3n}\lg(\tfrac{1}{\epsilon})(n d^{2n-2}\tau + \lg(\tfrac{1}{\epsilon})))
\]
bit operations with probability of success $1 - \tfrac{1}{\epsilon}$.

This bounds is the complexity of a Monte Carlo algorithm
for optimizing a multivariate polynomial over $\RR^{n}$.

By combining all the previous results, we have the following theorem.
\begin{theorem}\label{thm:unconstrained_detailed}
	Let $f \in \ZZ[x]$ be of degree $d$ and bitsize $\tau$.
	We can compute the optimum value of the problem $f^* = \inf{\x \in \RR^n} f(\x)$ in
	$\sOB(d^{3n}\lg(\tfrac{1}{\epsilon})(n d^{2n-2}\tau + \lg(\tfrac{1}{\epsilon})))$ bit operations, with probability of success $1 - \tfrac{1}{\epsilon}$.

In addition, it holds 
$2^{-\eta} \leq \abs{f^*} \leq 2^{\eta}$, where 
$\eta = \sOO(n d^{n-1}\tau + n^{2}))$. 
To distinguish $f^*$ from the other generalized critical values, we need approximate it with up to precision of at most 
$\sOO(n d^{2n-2}\tau + n^{2}d^{n-1})$ bits.
\end{theorem}

\bibliographystyle{abbrv}

\begin{thebibliography}{10}

\bibitem{bajaj1988algebraic}
C.~Bajaj.
\newblock The algebraic degree of geometric optimization problems.
\newblock {\em Discrete \& Computational Geometry}, 3:177--191, 1988.

\bibitem{BPR03}
S.~Basu, R.~Pollack, and {M-F.}Roy.
\newblock {\em Algorithms in Real Algebraic Geometry}, volume~10 of {\em
  Algorithms and Computation in Mathematics}.
\newblock Springer-Verlag, 2003.

\bibitem{bpr-06}
S.~Basu, R.~Pollack, and M.-F. Roy.
\newblock {\em Algorithms in Real Algebraic Geometry}.
\newblock Springer Berlin Heidelberg, 2006.

\bibitem{bd-cad-bound}
C.~W. Brown and J.~H. Davenport.
\newblock The complexity of quantifier elimination and cylindrical algebraic
  decomposition.
\newblock In {\em Proceedings of the 2007 international symposium on Symbolic
  and algebraic computation}, pages 54--60, 2007.

\bibitem{burr2020complexity}
M.~Burr, S.~Gao, and E.~Tsigaridas.
\newblock The complexity of subdivision for diameter-distance tests.
\newblock {\em Journal of Symbolic Computation}, 101:1--27, 2020.

\bibitem{canny_generalised_1990}
J.~Canny.
\newblock Generalised characteristic polynomials.
\newblock {\em Journal of Symbolic Computation}, 9(3):241--250, Mar. 1990.

\bibitem{caviness2012quantifier}
B.~F. Caviness and J.~R. Johnson.
\newblock {\em Quantifier elimination and cylindrical algebraic decomposition}.
\newblock Springer Science \& Business Media, 2012.

\bibitem{CLO2}
D.~A. Cox, J.~Little, and D.~O'shea.
\newblock {\em Using algebraic geometry}, volume 185.
\newblock Springer, 2006.

\bibitem{det-chow-24}
M.~L. Dogan, A.~A. Erg{\"u}r, and E.~Tsigaridas.
\newblock On the complexity of chow and hurwitz forms.
\newblock {\em ACM Communications in Computer Algebra}, 57(4):167--199, 2024.

\bibitem{draisma2016euclidean}
J.~Draisma, E.~Horobe{\c{t}}, G.~Ottaviani, B.~Sturmfels, and R.~R. Thomas.
\newblock The euclidean distance degree of an algebraic variety.
\newblock {\em Foundations of Computational Mathematics (FoCM)}, 16:99--149,
  2016.

\bibitem{Dur98}
A.~H. Durfee.
\newblock Five definitions of critical point at infinity.
\newblock pages 345--360. Springer, 1998.

\bibitem{elliott-bit-2020}
J.~Elliott, M.~Giesbrecht, and {\'E}.~Schost.
\newblock On the bit complexity of finding points in connected components of a
  smooth real hypersurface.
\newblock In {\em Proc. 45th International Symposium on Symbolic and Algebraic
  Computation (ISSAC)}, pages 170--177, 2020.

\bibitem{emt-dmm-j}
I.~Emiris, B.~Mourrain, and E.~Tsigaridas.
\newblock Separation bounds for polynomial systems.
\newblock {\em Journal of Symbolic Computation}, 101:128--151, 2020.

\bibitem{Est10}
A.~Esterov.
\newblock {N}ewton polyhedra of discriminants of projections.
\newblock {\em Discrete Comput. Geom.}, 44(1):96--148, 2010.

\bibitem{Est13}
A.~Esterov.
\newblock The discriminant of a system of equations.
\newblock {\em Advances in Mathematics}, 245:534--572, 2013.

\bibitem{farin2002curves}
G.~Farin and D.~Hansford.
\newblock {\em Curves and surfaces for CAGD: a practical guide}.
\newblock Morgan Kaufmann Publishers Inc., San Francisco, CA, USA, 5 edition,
  2002.

\bibitem{garg_special_2020}
A.~Garg.
\newblock Special case algorithms for {Nullstellensatz} and transcendence
  degree.
\newblock Master's thesis, IIT Kanpur, 2020.

\bibitem{GKZ94}
I.~M. Gel'fand, M.~M. Kapranov, and A.~V. Zelevinsky.
\newblock {\em Discriminants, resultants, and multidimensional determinants}.
\newblock Mathematics: Theory \& Applications. Birkh\"{a}user Boston, Inc.,
  Boston, MA, 1994.

\bibitem{GiHe93}
M.~Giusti and J.~Heintz.
\newblock La d{\'e}termination de la dimension et des points isol{\'e}es
  d’une vari{\'e}t{\'e} alg{\'e}brique peuvent s’ effectuer en temps
  polynomial.
\newblock In {\em Computational Algebraic Geometry and Commutative Algebra,
  Cortona}, volume~34, pages 216--256, 1991.

\bibitem{JK03}
Z.~Jelonek and K.~Kurdyka.
\newblock On asymptotic critical values of a complex polynomial.
\newblock {\em J. Reine Angew. Math.}, 565:1--11, 2003.

\bibitem{jelonek2005quantitative}
Z.~Jelonek and K.~Kurdyka.
\newblock Quantitative generalized bertini-sard theorem for smooth affine
  varieties.
\newblock {\em Discrete Comput. Geom.}, 34(4):659--678, 2005.

\bibitem{JelTib17}
Z.~Jelonek and M.~Tib{\u{a}}r.
\newblock Detecting asymptotic non-regular values by polar curves.
\newblock {\em International Mathematics Research Notices}, 2017(3):809--829,
  2017.

\bibitem{jeronimo2014probabilistic}
G.~Jeronimo and D.~Perrucci.
\newblock A probabilistic symbolic algorithm to find the minimum of a
  polynomial function on a basic closed semialgebraic set.
\newblock {\em Discrete \& Computational Geometry}, 52(2):260--277, 2014.

\bibitem{jpt-min-siopt-2013}
G.~Jeronimo, D.~Perrucci, and E.~Tsigaridas.
\newblock On the minimum of a polynomial function on a basic closed
  semialgebraic set and applications.
\newblock {\em SIAM Journal on Optimization}, 23(1):241--255, 2013.

\bibitem{Jean-Pierre_Jouanolou}
J.-P. Jouanolou.
\newblock {\em Th\'eor\`emes de Bertini et applications}.
\newblock Birkh\"auser, Boston, 1983.

\bibitem{khalil2002nonlinear}
H.~K. Khalil.
\newblock {\em Nonlinear systems}, volume~3.
\newblock Prentice hall Upper Saddle River, NJ, 2002.

\bibitem{kunz}
E.~Kunz.
\newblock {\em Introduction to commutative algebra and algebraic geometry}.
\newblock Springer Science \& Business Media, 1985.

\bibitem{la2024some}
C.~La~Valle and J.~Tonelli-Cueto.
\newblock Some lower bounds on the reach of an algebraic variety.
\newblock {\em arXiv e-prints}, pages arXiv--2402, 2024.

\bibitem{lipton}
R.~J. Lipton.
\newblock {The curious history of Schwarz-Zippel Lemma}.

\newblock {\em
  \url{}{https://rjlipton.wordpress.com
  /2009/11/30/the-curious-history-of
  -the-schwartz-zippel-lemma/}}, 2009.

\bibitem{mst-srus-17}
A.~Mantzaflaris, {\'E}.~Schost, and E.~Tsigaridas.
\newblock Sparse rational univariate representation.
\newblock In {\em Proceedings of the 2017 ACM on International Symposium on
  Symbolic and Algebraic Computation}, pages 301--308, 2017.

\bibitem{Nem86}
A.~N{\'e}methi.
\newblock Theorie de {L}efschetz pour vari{\'e}t{\'e}s algebriques affines.
\newblock {\em CR Acad. Sc. Paris}, 303:567--570, 1986.

\bibitem{NZ90}
A.~N{\'e}methi and A.~Zaharia.
\newblock On the bifurcation set of a polynomial function and newton boundary.
\newblock {\em Publications of the Research Institute for Mathematical
  Sciences}, 26(4):681--689, 1990.

\bibitem{Pan-usolve-02}
V.~Y. Pan.
\newblock Univariate polynomials: nearly optimal algorithms for numerical
  factorization and root-finding.
\newblock {\em Journal of Symbolic Computation}, 33(5):701--733, 2002.

\bibitem{pham2016genericity}
T.~S. Pham and H.~H. Vui.
\newblock {\em Genericity in polynomial optimization}, volume~3.
\newblock World Scientific, 2016.

\bibitem{rabier1997ehresmann}
P.~J. Rabier.
\newblock Ehresmann fibrations and {P}alais--{S}male conditions for morphisms
  of {F}insler manifolds.
\newblock {\em Ann. of Math.}, pages 647--691, 1997.

\bibitem{sharir}
O.~E. Raz, M.~Sharir, and J.~Solymosi.
\newblock "{P}olynomials vanishing on grids: The {E}lekes-{R}{\'o}nyai problem
  revisited".
\newblock In {\em Proceedings of the thirtieth annual symposium on
  Computational geometry}, page 251. ACM, 2014.

\bibitem{safey2007testing}
M.~Safey El~Din.
\newblock Testing sign conditions on a multivariate polynomial and
  applications.
\newblock {\em Mathematics in Computer Science}, 1(1):177--207, 2007.

\bibitem{Safey-opt-R-2008}
M.~Safey El~Din.
\newblock Computing the global optimum of a multivariate polynomial over the
  reals.
\newblock In {\em Proc. ACM 21st International Symposium on Symbolic and
  Algebraic Computation (ISSAC)}, pages 71--78, 2008.

\bibitem{saxena}
N.~Saxena.
\newblock Progress on polynomial identity testing.
\newblock {\em Bulletin of the EATCS}, 99:49--79, 2009.

\bibitem{schaefer2017fixed}
M.~Schaefer and D.~{\v{S}}tefankovi{\v{c}}.
\newblock Fixed points, nash equilibria, and the existential theory of the
  reals.
\newblock {\em Theory of Computing Systems}, 60(2):172--193, 2017.

\bibitem{schwartz1983piano}
J.~T. Schwartz and M.~Sharir.
\newblock On the “piano movers'” problem i. the case of a two-dimensional
  rigid polygonal body moving amidst polygonal barriers.
\newblock {\em Communications on pure and applied mathematics}, 36(3):345--398,
  1983.

\bibitem{sr-bag-94}
I.~R. Shafarevich and M.~Reid.
\newblock {\em Basic algebraic geometry}, volume~2.
\newblock Springer, 1994.

\bibitem{siciliano2009robotics}
B.~Siciliano, L.~Sciavicco, L.~Villani, and G.~Oriolo.
\newblock {\em Robotics: modelling, planning and control}.
\newblock Springer Science and Business Media, 2009.

\bibitem{song2017low}
Z.~Song, D.~P. Woodruff, and P.~Zhong.
\newblock Low rank approximation with entrywise l1-norm error.
\newblock In {\em Proc. 49th Annual ACM SIGACT Symposium on Theory of Computing
  (STOC)}, pages 688--701, 2017.

\bibitem{Tho69}
R.~Thom.
\newblock Ensembles et morphismes stratifi{\'e}s.
\newblock {\em Bull. Am. Math. Soc.}, 75(2):240--284, 1969.

\bibitem{Tib99}
M.~Tib{\u{a}}r.
\newblock Regularity at infinity of real and complex polynomial functions.
\newblock {\em London Mathematical Society Lecture Note Series}, pages
  249--264, 1999.

\bibitem{Tib07}
M.~Tib{\u{a}}r.
\newblock {\em Polynomials and vanishing cycles}, volume 170.
\newblock Cambridge University Press, 2007.

\bibitem{vdWaerden-alebra-37}
B.~L. Van~der Waerden.
\newblock Moderne algebra.
\newblock 1937.

\bibitem{varchenko1972theorems}
A.~N. Varchenko.
\newblock Theorems on the topological equisingularity of families of algebraic
  varieties and families of polynomial mappings.
\newblock {\em Izvestiya Rossiiskoi Akademii Nauk. Seriya Matematicheskaya},
  36(5):957--1019, 1972.

\bibitem{verdier1976stratifications}
J.-L. Verdier.
\newblock Stratifications de {W}hitney et th{\'e}oreme de {B}ertini-{S}ard.
\newblock {\em Invent. Math.}, 36(1):295--312, 1976.

\bibitem{vui2008critical}
H.~H. Vui and P.~T. Son.
\newblock Critical values of singularities at infinity of complex polynomials.
\newblock {\em Vietnam Journal of Mathematics}, 36(1):1--38, 2008.

\bibitem{wallace1971linear}
A.~Wallace.
\newblock Linear sections of algebraic varieties.
\newblock {\em Indiana Univ. Math. J.}, 20(12):1153--1162, 1971.

\bibitem{Yap-algebra-book}
C.-K. Yap.
\newblock {\em Fundamental problems of algorithmic algebra}.
\newblock Oxford University Press, New York, 2000.

\bibitem{Zah96}
A.~Zaharia.
\newblock On the bifurcation set of a polynomial function and {N}ewton
  boundary. {II}.
\newblock {\em Kodai Math. J.}, 19(2):218--233, 1996.

\end{thebibliography}
\def\cprime{$'$}

\appendix

\newpage
\section{Useful results and bounds}
\label{sec:useful-results-bounds}

\subsection{Multivariate polynomial multipication}

We need the following result(s) on multivariate polynomial multiplication. For the rather straightforward proofs we refer the reader to \cite{mst-srus-17}.

\begin{claim}[Output bounds on polynomial multiplication]
  \label{claim:mul-mpoly}
  The following bounds holds:
  \begin{enumerate}[(i)]
  \item Consider two multivariate polynomials, $f_1$ and $f_2$, in $\nu$
    variables of total degrees $\delta$, having bitsize $\tau_1$ and
    $\tau_2$, respectively.
    Then $f = f_1f_2$ is a polynomial in $\nu$ variables,
    of total degree $2 \delta$
    and bitsize $\tau_1 + \tau_2 + 2\, \nu \, \lg(\delta)$.
  \item Using induction, the product of $m$ polynomials in $\nu$ variables,
    $f = \prod_{i=1}^{m}f_i$, each of total degree  $\delta_i$
    and bitsize $\tau_i$, is a polynomial of
    total degree $\sum_{i=1}^{m}{\delta_i}$ and bitsize
    $\sum_{i=1}^{m}\tau_i + 12 \, \nu\, m\, \lg(m) \,
    \lg(\sum_{i=1}^{m}{\delta_i})$.

  \item Let $f$ be a polynomial in $\nu$ variables of total degree $\delta$
    and bitsize $\tau$.
    The $m$-th power of $f$, $f^m$, is a polynomial
    of total degree $m \delta$ and bitsize $m \tau + 12 \nu m \lg(\delta)$.
  \end{enumerate}
\end{claim}

\subsection{Intersecting a variety}

\begin{theorem}[{\cite[Cor.~1.13, Sec.~6, Chapter~1]{sr-bag-94}}]
  \label{thm:irr-X-inter-f}
  Consider an irreducible projective variety $V \subset \PP^n$ of dimension $\delta$.
  If  a homogeneous polynomial $f$ does not vanish on $V$, then
  $\dim(V \cap \VV(f)) = \delta -1$.
  In addition, all the components of the intersection have the same dimension.
\end{theorem}


The following lemma is similar to the ones in \cite[Chapter~3]{garg_special_2020}.

\begin{lemma}
  \label{lem:projV-intersect-H}
  Consider a projective variety $V \in \PP^n$
  of dimension $\delta$ and degree $D$.
  Let $S$ be finite set of integers such that $0 \not\in S$.
  Also consider the linear form $h(\x) = \sum_{i=0}^nc_i x_i$ and $H = \VV(h)$.
  If we choose the $c_i$'s independently and uniformly at random from $S$, then
  \[ \Pr[ \dim(V \cap H) = \delta - 1 ] \geq 1 - \tfrac{D}{\abs{S}}.
    \]

    Moreover, there is a specialization of the $c_i$'s  such that
    their bitsize is $\lceil \lg( \kappa D) \rceil + 1$, where $\kappa \geq 2$ is a constant,
    so that it holds $\dim(V \cap H) = \delta -1$.
\end{lemma}
\begin{proof}
  Let $V = \bigcup_{j=1}^r V_j$ be the irreducible decomposition of $V$,
  where $\deg(V_j) \geq 1$ and $\dim(V_j) = \delta_j$.
  Notice that $r \leq D$, as the number of components is bounded by the degree of the variety;
  in addition, $D = \sum_{i=j}^r \deg(V_j)$.
  Let $h(\x) = \sum_{i=0}^n c_i x_i$,
  where we consider the $c_i$'s as intermediates that we will specialize in the sequel.

  First, we consider the component $V_j$.
  Due to Theorem~\ref{thm:irr-X-inter-f}, $\dim(V_j \cap H) = \delta_j -1$
  or $\dim(V_j \cap H) = \delta_j$.
  In the latter case $V_j \subseteq H$ and so if $\p_j = (p_{j0}, \dots, p_{jn})  \in V_j$, then
  it holds $\p_j \in H$.

  We evaluate $h$ at $\p_j$, that is  $h_j = \sum_{i=0}^n c_i p_{ji}$
  and we consider $h_j$ as a polynomial in the $c_i$'s.
  Then $h_j$  has
  at most $\abs{S}^n$ solutions in the grid $\underbrace{S \times \cdots \times S}_{n+1} = S^{n+1}$.
  In other words there, are at most $\abs{S}^n$ specializations of the $c_i$'s
  in the grid $S^{n+1}$ such that $\dim(V_j \cap H) = \delta_j$.
  Consequently,
  \[ \Pr[\dim(V_j \cap H) = \delta_j] \leq \tfrac{\abs{S}^n}{\abs{S}^{n+1}}  = 1/\abs{S} .\]
  Thus
  \[
    \Pr[\bigcup_{j=1}^{r} \{ \dim(V_j \cap H) = \delta_j \} ] \leq
    \sum_{j=1}^{r} \Pr[\dim(V_j \cap H) = \delta_j] \leq
    \tfrac{r}{\abs{S}} \leq \tfrac{D}{\abs{S}},
  \]
  which implies
  \[
    \Pr[\dim(V \cap H) = \delta -1] =
    1 -  \Pr[\bigcup_{j=1}^{r} \{ \dim(V_j \cap H) = \delta_j \}]
    \geq 1 - \frac{D}{\abs{S}}.
  \]

  If we choose as $S$ the first $\kappa D$ positive integers, where $\kappa \geq 2$ is a constant,
  then there is at least one specialization of  the $c_i$'s in the grid $S^{n+1}$
  such that $\dim(V_j \cap H) = \delta_j - 1$, for all $j \in [r]$;
  and so the bound on the bitsize follows.
\end{proof}

%
\begin{proof}[Proof of Proposition~\ref{prp:affV-intersect-H}]
  In the affine case, we have to also consider the case $V \cap H = \emptyset$,
  in addition to the two cases of Theorem.~\ref{thm:irr-X-inter-f}.
  Let $\bar{V}$ and $\bar{H}$ be the projective closures of $V$ and $H$, respectively.
  It holds $\dim(\bar{V}) = \delta$ and $\deg(\bar{V}) = D$;
  also $\bar V \cap \bar H \not= \emptyset$, expect if $\dim(\bar V) = 0$.

  If $V \cap H = \emptyset$, then intersection $\bar{V} \cap \bar{H}$
  occurs at the hyperplane at infinity, say $L_{\infty}$.

  If we want to hold $\dim(V \cap H) = \delta =1$, then
  (i) $\dim(\bar V \cap \bar H) = \delta -1$ and
  (ii) $\bar V \cap \bar H \not\subseteq L_{\infty}$, unless $\bar V \cap \bar H = \emptyset$.

  \emph{The second condition implies
    if $V \cap H = \emptyset$, then $\bar V \cap \bar H$ is contained at infinity.
    Thus, the dimension of $\bar V \cap \bar H$ is $\dim(\bar V)$
    and because affine varieties and their projective closures have the same dimension,
    it also holds that $\dim(V \cap H) = \delta$.
    We want to avoid this situation.
    }

  For the first condition, as $\bar V$ and $\bar H$ are projective varieties,
  we can use Lemma~\ref{lem:projV-intersect-H}. Thus
  \[ \Pr  \geq 1 - \frac{D}{\abs{S}} \]
  and we can also select $c_i$'s.

  For the second condition, we notice that $\bar V \not\subseteq L_{\infty}$.
  Hence, by Theorem~\ref{thm:irr-X-inter-f}, $\dim(\bar \cap L_{\infty}) = \delta -1$.
  Now, consider the projective varieties
  $\bar V \cap L_{\infty}$ and $\bar H$.

  If it happens, following Lemma~\ref{lem:projV-intersect-H} that
  \[
    \Pr[ \dim( \bar V \cap L_{\infty} \cap \bar H) = \delta -2 ] \geq 1 - \frac{D}{\abs{S}}
  \]
  and we can also find appropriate $c_i$'s.

  If it happens that $\dim( \bar V \cap L_{\infty} \cap \bar H) = \delta -2$,
  then  $\bar V \cap \bar H \not\subseteq L_{\infty}$.
  This is so, because otherwise
  $\bar V \cap L_{\infty} \cap \bar H = \bar V \cap \bar H$,
  where the dimensions are $\delta -2$ and $\delta -1$, respectively.
  This is a contradiction.

  Using the union bound, we deduce the required probability bound.
\end{proof}

\subsection{Reduction to a square system}
\label{sec:make-square-sys}

Our techniques on bounding the bitsize and the degrees of the various
polynomials appearing in the polynomial optimization problems rely 
on exploiting the properties of the resultant. However, 
to use the resultant the corresponding polynomial systems have to have as many equations as variables; that is we need to compute with square systems. 
In the optimization problems that we are interested in, in
almost all the cases, the number of polynomials is bigger than the
number of unknowns. 

We are able to treat this case using a result due to Giusti and Heinz
\cite{GiHe93}. Let $f_1,\dots,f_p\in\Q\left[x_1,\dots,x_n\right]$ be
polynomials of positive degree, bounded by $d$. Denote by $V$ the algebraic
variety $\VV(f_1,\dots,f_p) \subseteq \CC^n$. Given $\eta\in\overline{\Q}$, we denote
by $\hat{f}_\eta$ the linear combination
$f_1+\eta^1f_2+\dots+\eta^{p-1}f_p$.

\begin{theorem}
  \cite[Section 3.4.1]{GiHe93}
  \label{thm:make-square}
  Let $\Gamma\subset\Z$ of cardinal $p\, d^n+1$. There exists
  $\gamma=\left(\gamma_1,\dots,\gamma_n\right)\in\Gamma^n$ such that
  each irreducible component of $\hat{V} =
  \VV(\hat{f}_{\gamma_1},\dots, \hat{f}_{\gamma_n})$ is either a
  component of $V$ or a point.
\end{theorem}
\begin{proof}  
  First, it is clear that for all $x\in V$, any linear combination of
  the polynomials $f_1,\dots,f_p$ vanishes at $x$.

  Then for $1\leq i\leq n$ and
  $\left(\gamma_1,\dots,\gamma_i\right)\in\Gamma^i$, let us denote by
  $\hat{V}_{i}$ the variety
  $\VV(\hat{f}_{\gamma_1},\dots,\hat{f}_{\gamma_{i}})$.  We prove
  by induction that for $1\leq i\leq n$, there exist points
  $\gamma_1,\dots,\gamma_i\in\Gamma$ such that the dimension of any
  irreducible component of $\hat{V}_{i}$ not contained in $V$ is
  $n-i$.

  The case $i=1$ is obvious. Let $i>1$ and assume that the result is
  proved for $i-1$. Let
  $\hat{V}_{i-1}=\VV(\hat{f}_{\gamma_1},\dots,\hat{f}_{\gamma_{i-1}})$
  and let $C$ be an irreducible component of $\hat{V}_{i-1}$ such that
  $C\not\subset V$. Let $x_C$ be a point in $C\setminus V$, so that at
  least one polynomial among $f_1,\dots,f_p$ does not vanish at
  $x_C$. Let $T$ be a new indeterminate and let
  \[w_C\left(T\right)=f_1\left(x_C\right) + T f_2\left(x_C\right) + \dots + T^{p-1} f_p\left(x_C\right) \in
  \overline{\Q}\left[T\right].\] 
  By construction, $w_C$ is not identically zero. Hence, so is the
  product $P_i\left(T\right) = \prod_{C} w_C\left(T\right)$, 
  for all irreducible component $C$ of
  $\hat{V}_{i-1}$ not included in $V$. By the B\'ezout bound, there
  are at most $d^n$ such components. Hence, the polynomial
  $P_i\left(T\right)$ has degree at most $pd^n$ so that it has at most
  $pd^n$ roots. Since $\Gamma$ has cardinal $pd^n+1$, there exists
  $\gamma_i\in\Gamma$ such that $P_i\left(\gamma_i\right)\neq 0$. Let
  $\hat{f}_{\gamma_i} = f_1 + \gamma_i f_2 + \dots + \gamma_i^{s-1}f_p$.

  Then for any $x_C$ as above, $\hat{f}_{\gamma_i}\left(x_C\right)\neq 0$. 
  In particular, this means that
  $C\not\subset\VV(\hat{f}_{\gamma_i})$.

  By Krull's Principal Ideal Theorem \cite[Corollary 3.2
  p. 131]{kunz}, this implies that the intersection
  $C\cap\VV(\hat{f}_{\gamma_i})$ is either empty or
  equidimensional of dimension $\dim C - 1$. By induction assumption,
  $C$ has dimension $n-i+1$. Hence,
  $C\cap\VV(\hat{f}_{\gamma_i})$ is either empty or of
  equidimensional of dimension $n-i$. Since this is true for any
  irreducible component $C$ of $\hat{V}_{i-1}$ not contained in $V$,
  this proves that the dimension of any irreducible component of
  $\hat{V}_{i}$ not contained in $V$ is $n-i$.

  In particular for $i=n$, this proves that the irreducible components
  of $\hat{V}=\hat{V}_{n}$ not contained in $V$ have dimension $0$, so
  that it is necessarily a point.
\end{proof}

\begin{remark}
	\label{rem:make-square}
	The previous theorem implies that if we are given a polynomial system that it is not necessarily square, then we can make square by considering a generic linear combination. 
	The price that we pay for this is that the bitsize of the polynomials becomes $\OO(\tau + n \lg(p d)) = \sOO(\tau + n)$ from $\tau$ and that we add additional points to the variety. 
	For out purposes, the increase in the in bitsize is negligible. 
\end{remark}

\end{document}